\documentclass[a4paper,10pt]{amsart}
\usepackage{amsmath,amssymb,color,multirow}
\usepackage[colorlinks]{hyperref}
\definecolor{linkblue}{RGB}{1,1,190}
\definecolor{citered}{RGB}{190,1,1}
\hypersetup{linkcolor=linkblue,urlcolor=linkblue,citecolor=citered}

\theoremstyle{plain}
\newtheorem{theorem}{\bf Theorem}[section]
\newtheorem{proposition}[theorem]{\bf Proposition}
\newtheorem{lemma}[theorem]{\bf Lemma}

\theoremstyle{definition}

\newtheorem{conjecture}[theorem]{\bf Conjecture}

\hypersetup{hypertexnames=false}
\numberwithin{equation}{section}
\makeatletter\@namedef{subjclassname@2020}{\textup{2020} Mathematics Subject Classification}\makeatother
\allowdisplaybreaks

\begin{document}
\title[On counterexamples to Mordell's Conjecture and the AAC-Conjecture]{On counterexamples to Mordell's Pellian Equation Conjecture and the AAC-Conjecture: a non-computer based approach}
\author{Andreas Reinhart}
\address{Institut f\"ur Mathematik und Wissenschaftliches Rechnen, Karl-Franzens-Universit\"at Graz, NAWI Graz, Heinrichstra{\ss}e 36, 8010 Graz, Austria}
\email{andreas.reinhart@uni-graz.at}
\thanks{This work was supported by the Austrian Science Fund FWF, Project Number P36852-N}

\begin{abstract}
In this note, we discuss recently discovered counterexamples to Mordell's Pellian Equation Conjecture and the Ankeny-Artin-Chowla-Conjecture. We provide a verification of the counterexample to Mordell's Pellian Equation Conjecture that can be checked with marginal computer assistance.
\end{abstract}

\maketitle

\section{Introduction and main result}\label{1}

We fix notation which is needed to formulate the Ankeny-Artin-Chowla-Conjecture and Mordell's Pellian Equation Conjecture.

\smallskip
Let $\mathbb{P}$, $\mathbb{N}$, $\mathbb{N}_0$, $\mathbb{Z}$, $\mathbb{Q}$ denote the sets of prime numbers, positive integers, nonnegative integers, integers and rational numbers, respectively. For each $p\in\mathbb{P}$ and each $a\in\mathbb{Z}\setminus\{0\}$, let ${\rm v}_p\left(a\right)$ be the $p$-adic exponent of $a$ (i.e., the largest $k\in\mathbb{N}_0$ with $p^k\mid a$). Let $d\in\mathbb{N}_{\geq 2}$ be squarefree, let $K=\mathbb{Q}\left(\sqrt{d}\right)$, let $\mathcal{O}_K$ be the ring of algebraic integers of $K$, let $\mathcal{O}_d=\mathbb{Z}+d\mathcal{O}_K$ and let $\varepsilon>1$ be the fundamental unit of $\mathcal{O}_K$. By $\mathcal{O}_K^{\times}$ (respectively $\mathcal{O}_d^{\times}$) we denote the unit group of $\mathcal{O}_K$ (respectively of $\mathcal{O}_d$). Let ${\rm N}:K\rightarrow\mathbb{Q}$ defined by ${\rm N}\left(a+b\sqrt{d}\right)=a^2-db^2$ for each $a,b\in\mathbb{Q}$ be the norm map. Moreover, for each nonzero ideal $I$ of $\mathcal{O}_K$, let ${\rm N}\left(I\right)=|\mathcal{O}_K/I|$ be the ideal norm of $I$. We set

\[
\omega=\begin{cases}\sqrt{d}&\textnormal{if }d\equiv 2,3\mod 4,\\\frac{1+\sqrt{d}}{2}&\textnormal{if }d\equiv 1\mod 4.\end{cases}
\]

\smallskip
Observe that $\mathcal{O}_K=\mathbb{Z}[\omega]=\mathbb{Z}\oplus\omega\mathbb{Z}$. Let $p\in\mathbb{P}$. Then $p$ is called {\it ramified} in $\mathcal{O}_K$ if $p\mathcal{O}_K$ is the square of a maximal ideal of $\mathcal{O}_K$. Moreover, we say that $p$ {\it splits} in $\mathcal{O}_K$ if $p\mathcal{O}_K$ is the product of two distinct maximal ideals of $\mathcal{O}_K$.

We recall two (long standing) conjectures.

\begin{conjecture}[Ankeny-Artin-Chowla-Conjecture]\label{Conjecture 1.1} Let $x,y\in\mathbb{N}_0$ be such that $\varepsilon=x+y\omega$. If $d\in\mathbb{P}$ and $d\equiv 1\mod 4$, then $d\nmid y$.
\end{conjecture}

\begin{conjecture}[Mordell's Pellian Equation Conjecture]\label{Conjecture 1.2} Let $x,y\in\mathbb{N}_0$ be such that $\varepsilon=x+y\omega$. If $d\in\mathbb{P}$ and $d\equiv 3\mod 4$, then $d\nmid y$.
\end{conjecture}

We briefly discuss the history and the significance of both conjecture. The Ankeny-Artin-Chowla-Conjecture was first mentioned in 1952 (see \cite[page 480]{An-Ar-Ch52}) and it was subsequently investigated in a multitude of papers \cite{Ag16,Be-Mo24,Ha-Jo18,Ha01,Lu92,Mo60,Si-Sh24,St-Wi88}. Mordell's Pellian Equation Conjecture was stated by L. J. Mordell in 1961 (see \cite[page 283]{Mo61}) and independently by A. A. Kiselev and I. Sh. Slavutski\u{\i} in 1959 (see \cite{Ki-Sl59}). It was investigated in various papers and by many authors (for instance, see \cite{Be-Mo24,Ch-Sa19,Lu91,Si-Sh24,St-Wi88}).

\smallskip
Both conjectures were recently disproved by counterexamples which were found by a long numerical computer search \cite{Re24,Rei24}. For the methods, algorithms and programs that were used to find (and check) these counterexamples, we refer to \cite[Section 3]{Re24} and \cite[Section 2]{Rei24}. Note that the algorithms (that were used in \cite{Re24,Rei24} to check the counterexamples) require a multitude of computational steps and make (at least to some extent) use of continued fraction expansions. For instance, the fundamental unit of $\mathcal{O}_K$ for $d$ in Theorem~\ref{Theorem 1.3} is larger than $10^{1880000}$ and the period length of the continued fraction expansion of $\sqrt{d}$ is $3650856$.

\smallskip
The purpose of this note is to demonstrate how to verify the given counterexamples with manageable effort and only minor computer assistance. We do this in detail for the counterexample to Mordell's Pellian Equation Conjecture. The verification of the counterexample to the AAC-Conjecture can be done along similar lines. Also note that we do not use continued fraction expansions, and thus we get additional independent verifications of the counterexamples.

\smallskip
Note that both the Ankeny-Artin-Chowla-Conjecture and Mordell's Pellian Equation Conjecture have ties with some problems in (commutative) algebra. Let $\mathcal{O}$ be an order in $K$. Then $\mathcal{O}$ is said to have {\it torsionfree cancellation} if for all torsionfree finitely generated $\mathcal{O}$-modules $A,B,C$ with $A\oplus B\cong A\oplus C$, it follows that $B\cong C$. Moreover, $\mathcal{O}$ is said to have {\it mixed cancellation} if for all finitely generated $\mathcal{O}$-modules $A,B,C$ with $A\oplus B\cong A\oplus C$, we have that $B\cong C$. Let $d\in\mathbb{P}$ and let $x,y\in\mathbb{N}_0$ be such that $\varepsilon=x+y\omega$. It is well-known that $\mathcal{O}_d$ has torsionfree cancellation if and only if $d\nmid y$ \cite[Theorem 4.9]{W84}. In particular, it follows from \cite[Theorem 4.9]{W84} and \cite[Theorem 5.12]{Ha07} that if $d\geq 7$, then $\mathcal{O}_d$ has torsionfree cancellation, but not mixed cancellation if and only $d\nmid y$.

\smallskip
We also want to mention a problem from factorization theory that is connected to the condition $d\mid y$. It was recently studied in \cite{Re23,Re24}. For a profound introduction to factorization theory see \cite{Ge-HK06a}. Let $\mathcal{O}$ be an order in $K$. A nonzero nonunit of $\mathcal{O}$ that cannot be written as a product of two nonunits of $\mathcal{O}$ is called an {\it atom} of $\mathcal{O}$. It is well-known that every nonzero nonunit of $\mathcal{O}$ is a finite product of atoms of $\mathcal{O}$ (see \cite{Ge-HK06a}). Let $x\in\mathcal{O}$ be a nonzero nonunit. If $\ell\in\mathbb{N}$ is such that $x$ is a product of $\ell$ atoms of $\mathcal{O}$, then $\ell$ is called a {\it length} of $x$. Let $\mathsf{L}(x)$ be the set of lengths of $x$. Set $L=\min\{r-s\mid y\in\mathcal{O}$ is a nonzero nonunit, $r,s\in\mathsf{L}(y), s<r\}$ (note that we use $\min\emptyset=0$ here for convenience). It is known (see \cite{Re23}) that $L\leq 2$. We say that $\mathcal{O}$ is {\it unusual} if $L=2$. The condition $d\mid y$ determines in many cases whether $K$ has an unusual order (see \cite[Proposition 2.4 and Theorem 2.5]{Re24}).

\smallskip
For a survey on the Conjectures~\ref{Conjecture 1.1} and~\ref{Conjecture 1.2} and further discussion of their importance, we refer to \cite{Sl04}. For a more detailed study of the property $d\mid y$ for arbitrary squarefree $d\in\mathbb{N}_{\geq 2}$, see \cite{Re24}. Next, we present the main result of this note.

\begin{theorem}\label{Theorem 1.3}
The number $d=39028039587479$ is a counterexample to Mordell's Pellian Equation Conjecture.
\end{theorem}

The proof of Theorem~\ref{Theorem 1.3} in Section~\ref{3} was created with computer assistance (and with the help of the programs Mathematica 12.0.0 and Pari/GP 2.15.2). We used the command bnfunits(bnfinit(X${}^\wedge$2-d,1)) in Pari/GP (together with some elementary algebraic manipulations) to obtain the values in the table (with the entries $a_i$, $b_i$, $c_i$, $d_i$, $e_i$ and $f_i$) below.

Observe that (very dedicated) arithmeticians may be able to verify this proof without computer assistance. We define an arithmetical operation as {\it hard} if it is a multiplication of integers $a$ and $b$ with $|a|\geq 10^3$, $|b|\geq 10^3$ and $|ab|\geq 10^{10}$. All other arithmetical operations in this note (additions, subtractions and multiplications of integers) are defined to be {\it easy}. We define additions and subtractions where (at least) one of the members is a single digit integer as {\it trivial}. A rough estimate yields that there are approximately 550 hard operations and about 2200 (nontrivial) easy operations in the proof of Theorem~\ref{Theorem 1.3} below. If we assume that (the verification of) a hard operation takes about 15 minutes (on average) and (the verification of) an easy (but nontrivial) operation takes about 3 minutes (on average), then a verification of the proof below is doable in about 250 hours (without computers). If all arithmetical operations are done with computers, then it should only take a modest amount of time (to check all the details).

\section{Preliminaries}\label{2}

We recall a variant of a result of H. C. Pocklington (cf. \cite[Theorem 4]{BrLeSe75} and \cite{Po14}) that will be used to show that $d$ is prime. For the readers' convenience we include a proof. This proof goes along the lines of the proof of \cite[Theorem 4 and Corollary 1]{BrLeSe75}.

\begin{proposition}\label{Proposition 2.1}
Let $a,b,c\in\mathbb{N}$ be such that $a^2>c>1$, $a\mid c-1$, $b^{c-1}\equiv 1\mod c$ and for each $p\in\mathbb{P}$ with $p\mid a$, ${\rm gcd}\left(b^{\frac{c-1}{p}}-1,c\right)=1$. Then $c\in\mathbb{P}$.
\end{proposition}

\begin{proof}
\textsc{Claim}: For each $q\in\mathbb{P}$ with $q\mid c$, it follows that $q>a$.

\smallskip
Let $q\in\mathbb{P}$ be such that $q\mid c$. Then $q\nmid b$ (since $b^{c-1}\equiv 1\mod c$). Set $e=\min\{k\in\mathbb{N}\mid b^k\equiv 1\mod q\}$. We obtain that $e\mid q-1$. It remains to show that ${\rm v}_p\left(a\right)\leq {\rm v}_p\left(e\right)$ for each $p\in\mathbb{P}$ with $p\mid a$. (Then $a\mid e\mid q-1$, and hence $q=1+a\ell$ for some $\ell\in\mathbb{N}$, which implies that $q>a$.) Let $p\in\mathbb{P}$ be such that $p\mid a$. Since $b^{c-1}\equiv 1\mod c$, we have that $b^{c-1}\equiv 1\mod q$, and thus $e\mid c-1$. Moreover, since ${\rm gcd}\left(b^{\frac{c-1}{p}}-1,c\right)=1$, we infer that $q\nmid b^{\frac{c-1}{p}}-1$. Therefore, $b^{\frac{c-1}{p}}\not\equiv 1\mod q$, and hence $e\nmid\frac{c-1}{p}$. In particular, we obtain that $p\nmid\frac{c-1}{e}$, and thus ${\rm v}_p\left(c-1\right)={\rm v}_p\left(\frac{c-1}{e}\right)+{\rm v}_p\left(e\right)={\rm v}_p\left(e\right)$. Since $a\mid c-1$, we have that ${\rm v}_p\left(a\right)\leq {\rm v}_p\left(e\right)$.\qed(Claim)

\smallskip
Assume to the contrary that $c\not\in\mathbb{P}$. Then there are some $r,s\in\mathbb{P}$ such that $rs\mid c$. We infer by the claim that $c\geq rs>a^2>c$, a contradiction.
\end{proof}

Next we mention a simple (well-known) lemma that will be used to compute integer powers modulo $d$.

\begin{lemma}\label{Lemma 2.2}
Let $a,n\in\mathbb{N}$, let $m\in\mathbb{N}_0$ and let $\left(g_i\right)_{i=0}^m$ be the sequence of integers with $g_i\in\{0,1\}$ for each $i\in [0,m]$, $g_m=1$ and $n=\sum_{j=0}^m g_j2^j$ $($i.e., $\left(g_i\right)_{i=0}^m$ is the binary representation of $n)$. Furthermore, let $\left(h_i\right)_{i=0}^{m+1}$ be defined recursively by $h_0=1$ and $h_{r+1}=a^{g_{m-r}}h_r^2$ for each $r\in [0,m]$. Then $h_{m+1}=a^n$.
\end{lemma}

\begin{proof}
It is straightforward to show by induction that $h_{r+1}=a^{\sum_{j=0}^r g_{m-j}2^{r-j}}$ for each $r\in [0,m]$. We obtain that $h_{m+1}=a^{\sum_{j=0}^m g_{m-j}2^{m-j}}=a^{\sum_{j=0}^m g_j2^j}=a^n$.
\end{proof}

\section{Proof of Theorem~\ref{Theorem 1.3}}\label{3}

In what follows, we will use the sieve of Eratosthenes and well-known facts about $\mathcal{O}_K$ (e.g. that $\mathcal{O}_K$ is a Dedekind domain) without further mention. Besides that, we will indicate the binary representation of a positive integer by using the subscript ``${}_2$''.

\smallskip
Let $d=39028039587479$ and let $x,y\in\mathbb{N}_0$ be such that $\varepsilon=x+y\omega$. We need to show that $d\equiv 3\mod 4$, $d\in\mathbb{P}$ and $d\mid y$. Since $79=19\cdot 4+3$, we have that $d\equiv 79\equiv 3\mod 4$. Observe that $\omega=\sqrt{d}$, $x,y\in\mathbb{N}$ and $\varepsilon=x+y\sqrt{d}$.

\medskip
\textsc{Claim} 1: $d\in\mathbb{P}$.

\medskip
Set $p=3617$, $q=4021$ and $a=pq$. By Proposition~\ref{Proposition 2.1}, it is sufficient to show that $p,q\in\mathbb{P}$, $a^2>d$, $a\mid d-1$, $2^{d-1}\equiv 1\mod d$ and ${\rm gcd}\left(2^{\frac{d-1}{p}}-1,d\right)={\rm gcd}\left(2^{\frac{d-1}{q}}-1,d\right)=1$.

\medskip
Note that $\{t\in\mathbb{P}\mid t^2\leq p\}=\{2,3,5,7,11,13,17,19,23,29,31,37,41,43,47,53,59\}$ and

{\tiny
\begin{align*}
p=&1808\cdot 2+1=1205\cdot 3+2=723\cdot 5+2=516\cdot 7+5=328\cdot 11+9=278\cdot 13+3=212\cdot 17+13=190\cdot 19+7=157\cdot 23+6\\
=&124\cdot 29+21=116\cdot 31+21=97\cdot 37+28=88\cdot 41+9=84\cdot 43+5=76\cdot 47+45=68\cdot 53+13=61\cdot 59+18.
\end{align*}
}

This implies that $p\in\mathbb{P}$.

\medskip
Moreover, $\{t\in\mathbb{P}\mid t^2\leq q\}=\{2,3,5,7,11,13,17,19,23,29,31,37,41,43,47,53,59,61\}$ and

{\tiny
\begin{align*}
q=&2010\cdot 2+1=1340\cdot 3+1=804\cdot 5+1=574\cdot 7+3=365\cdot 11+6=309\cdot 13+4=236\cdot 17+9=211\cdot 19+12=174\cdot 23+19\\
=&138\cdot 29+19=129\cdot 31+22=108\cdot 37+25=98\cdot 41+3=93\cdot 43+22=85\cdot 47+26=75\cdot 53+46=68\cdot 59+9=65\cdot 61+56.
\end{align*}
}

Consequently, $q\in\mathbb{P}$. Observe that $a^2>\left(3000\cdot 4000\right)^2>10^{14}>d$ and $d=1+10790168534p=1+9706053118q=1+2683454a$. In particular, we obtain that $a\mid d-1$.

\medskip
Next, we show that $2^{d-1}\equiv 1\mod d$ by applying Lemma~\ref{Lemma 2.2}. Each item in the list of equations below consists of two parts. The first part is used to derive the binary representation of $d-1$, while the second part is the outcome of determining the representatives of the integers $h_i$ (in Lemma~\ref{Lemma 2.2}) modulo $d$. Also note that the second part contains an additional factor $2$ if and only if the remainder (modulo $2$) in the first part is $1$. (Keep in mind that the computation of the binary representation of $d-1$ is technically done from the bottom to the top.)

\medskip

{\tiny
$\bullet$ $2\cdot 0+1=1$,\hspace*{35.1mm} $2\cdot 1^2=2$.

$\bullet$ $2\cdot 1=2$,\hspace*{39.5mm} $2^2=4$.

$\bullet$ $2\cdot 2=4$,\hspace*{39.5mm} $4^2=16$.

$\bullet$ $2\cdot 4=8$,\hspace*{39.5mm} $16^2=256$.

$\bullet$ $2\cdot 8+1=17$,\hspace*{33.8mm} $2\cdot 256^2=131072$.

$\bullet$ $2\cdot 17+1=35$,\hspace*{32.5mm} $2\cdot 131072^2=34359738368$.

$\bullet$ $2\cdot 35=70$,\hspace*{36.9mm} $34359738368^2=30249831d+18934861837375$.

$\bullet$ $2\cdot 70+1=141$,\hspace*{31.2mm} $2\cdot 18934861837375^2=18372892750447d+8878320928137$.

$\bullet$ $2\cdot 141+1=283$,\hspace*{29.9mm} $2\cdot 8878320928137^2=4039382112766d+4872136924624$.

$\bullet$ $2\cdot 283+1=567$,\hspace*{29.9mm} $2\cdot 4872136924624^2=1216444303284d+37954390101716$.

$\bullet$ $2\cdot 567+1=1135$,\hspace*{28.6mm} $2\cdot 37954390101716^2=73820552772801d+14864000930633$.

$\bullet$ $2\cdot 1135+1=2271$,\hspace*{27.3mm} $2\cdot 14864000930633^2=11322040563715d+5306939836893$.

$\bullet$ $2\cdot 2271+1=4543$,\hspace*{27.3mm} $2\cdot 5306939836893^2=1443250069954d+7254835280932$.

$\bullet$ $2\cdot 4543=9086$,\hspace*{31.7mm} $7254835280932^2=1348585158511d+1741096904855$.

$\bullet$ $2\cdot 9086+1=18173$,\hspace*{26mm} $2\cdot 1741096904855^2=155345667583d+18644662148793$.

$\bullet$ $2\cdot 18173+1=36347$,\hspace*{24.7mm} $2\cdot 18644662148793^2=17814034746144d+34365896782722$.

$\bullet$ $2\cdot 36347+1=72695$,\hspace*{24.7mm} $2\cdot 34365896782722^2=60521352041448d+19084525628976$.

$\bullet$ $2\cdot 72695=145390$,\hspace*{27.8mm} $19084525628976^2=9332242211824d+32132914656880$.

$\bullet$ $2\cdot 145390+1=290781$,\hspace*{22.1mm} $2\cdot 32132914656880^2=52911917444994d+20520829038674$.

$\bullet$ $2\cdot 290781+1=581563$,\hspace*{22.1mm} $2\cdot 20520829038674^2=21579583749811d+29933591140083$.

$\bullet$ $2\cdot 581563=1163126$,\hspace*{25.2mm} $29933591140083^2=22958362449471d+9944907473280$.

$\bullet$ $2\cdot 1163126=2326252$,\hspace*{23.9mm} $9944907473280^2=2534105881245d+7176351027045$.

$\bullet$ $2\cdot 2326252=4652504$,\hspace*{23.9mm} $7176351027045^2=1319564462056d+17827533235201$.

$\bullet$ $2\cdot 4652504+1=9305009$,\hspace*{19.5mm} $2\cdot 17827533235201^2=16286800188352d+30177590176194$.

$\bullet$ $2\cdot 9305009+1=18610019$,\hspace*{18.2mm} $2\cdot 30177590176194^2=46668341964810d+29055194037282$.

$\bullet$ $2\cdot 18610019=37220038$,\hspace*{21.3mm} $29055194037282^2=21630712417718d+35791355394602$.

$\bullet$ $2\cdot 37220038+1=74440077$,\hspace*{16.9mm} $2\cdot 35791355394602^2=65646193584044d+32203984891732$.

$\bullet$ $2\cdot 74440077+1=148880155$,\hspace*{15.6mm} $2\cdot 32203984891732^2=53146233009337d+18209288628225$.

$\bullet$ $2\cdot 148880155+1=297760311$,\hspace*{14.3mm} $2\cdot 18209288628225^2=16991793379874d+22534432303604$.

$\bullet$ $2\cdot 297760311+1=595520623$,\hspace*{14.3mm} $2\cdot 22534432303604^2=26022349296203d+5977961735395$.

$\bullet$ $2\cdot 595520623+1=1191041247$,\hspace*{13mm} $2\cdot 5977961735395^2=1831300105645d+26065072393095$.

$\bullet$ $2\cdot 1191041247=2382082494$,\hspace*{16.1mm} $26065072393095^2=17407689600562d+2873437115827$.

$\bullet$ $2\cdot 2382082494=4764164988$,\hspace*{16.1mm} $2873437115827^2=211556638403d+25924124537892$.

$\bullet$ $2\cdot 4764164988+1=9528329977$,\hspace*{11.8mm} $2\cdot 25924124537892^2=34439866319688d+247631620776$.

$\bullet$ $2\cdot 9528329977=19056659954$,\hspace*{14.8mm} $247631620776^2=1571214446d+9241789320542$.

$\bullet$ $2\cdot 19056659954+1=38113319909$,\hspace*{9.1mm} $2\cdot 9241789320542^2=4376887527432d+9150470123600$.

$\bullet$ $2\cdot 38113319909+1=76226639819$,\hspace*{9.1mm} $2\cdot 9150470123600^2=4290817800121d+10455437635041$.

$\bullet$ $2\cdot 76226639819=152453279638$,\hspace*{12.2mm} $10455437635041^2=2800965082942d+24962411388463$.

$\bullet$ $2\cdot 152453279638+1=304906559277$,\hspace*{6.5mm} $2\cdot 24962411388463^2=31932015490051d+31490910333309$.

$\bullet$ $2\cdot 304906559277=609813118554$,\hspace*{10.9mm} $31490910333309^2=25409358094908d+8143824432549$.

$\bullet$ $2\cdot 609813118554=1219626237108$,\hspace*{9.6mm} $8143824432549^2=1699339169714d+30354920226395$.

$\bullet$ $2\cdot 1219626237108+1=2439252474217$,\hspace*{3.9mm} $2\cdot 30354920226395^2=47218419971389d+29633566753719$.

$\bullet$ $2\cdot 2439252474217=4878504948434$,\hspace*{8.3mm} $29633566753719^2=22500445521451d+15832914818932$.

$\bullet$ $2\cdot 4878504948434+1=9757009896869$,\hspace*{3.9mm} $2\cdot 15832914818932^2=12846209766784d+7728534743712$.

$\bullet$ $2\cdot 9757009896869+1=19514019793739$,\hspace*{2.6mm} $2\cdot 7728534743712^2=3060889038553d+1$.

$\bullet$ $2\cdot 19514019793739=39028039587478$,\hspace*{5.7mm} $1^2=1$.

}

\medskip
This shows that $d-1=39028039587478=1000110111111011101100011011111001011010010110_2$. By Lemma~\ref{Lemma 2.2}, we obtain that $2^{d-1}\equiv 1\mod d$. Now we show that $2^{\frac{d-1}{p}}\equiv 10285064380914\mod d$ by applying Lemma~\ref{Lemma 2.2} again. This is done along the lines of the proof of $2^{d-1}\equiv 1\mod d$ above.

\medskip

{\tiny
$\bullet$ $2\cdot 0+1=1$,\hspace*{35.1mm} $2\cdot 1^2=2$.

$\bullet$ $2\cdot 1=2$,\hspace*{39.5mm} $2^2=4$.

$\bullet$ $2\cdot 2+1=5$,\hspace*{35.1mm} $2\cdot 4^2=32$.

$\bullet$ $2\cdot 5=10$,\hspace*{38.2mm} $32^2=1024$.

$\bullet$ $2\cdot 10=20$,\hspace*{36.9mm} $1024^2=1048576$.

$\bullet$ $2\cdot 20=40$,\hspace*{36.9mm} $1048576^2=1099511627776$.

$\bullet$ $2\cdot 40=80$,\hspace*{36.9mm} $1099511627776^2=30975827440d+31390886082416$.

$\bullet$ $2\cdot 80=160$,\hspace*{35.6mm} $31390886082416^2=25248199485668d+897251646084$.

$\bullet$ $2\cdot 160+1=321$,\hspace*{29.9mm} $2\cdot 897251646084^2=41255493481d+23445371345713$.

$\bullet$ $2\cdot 321+1=643$,\hspace*{29.9mm} $2\cdot 23445371345713^2=28168744489781d+6603555904639$.

$\bullet$ $2\cdot 643=1286$,\hspace*{33mm} $6603555904639^2=1117323622877d+23588814563238$.

$\bullet$ $2\cdot 1286=2572$,\hspace*{31.7mm} $23588814563238^2=14257241162513d+36673782069917$.

$\bullet$ $2\cdot 2572+1=5145$,\hspace*{27.3mm} $2\cdot 36673782069917^2=68923077127515d+7453272389093$.

$\bullet$ $2\cdot 5145=10290$,\hspace*{30.4mm} $7453272389093^2=1423368170504d+28404889843233$.

$\bullet$ $2\cdot 10290=20580$,\hspace*{29.1mm} $28404889843233^2=20673284529132d+20389277954061$.

$\bullet$ $2\cdot 20580+1=41161$,\hspace*{24.7mm} $2\cdot 20389277954061^2=21303793881634d+11174538322756$.

$\bullet$ $2\cdot 41161=82322$,\hspace*{29.1mm} $11174538322756^2=3199502410231d+12534158337887$.

$\bullet$ $2\cdot 82322=164644$,\hspace*{27.8mm} $12534158337887^2=4025442397307d+8237359105716$.

$\bullet$ $2\cdot 164644+1=329289$,\hspace*{22.1mm} $2\cdot 8237359105716^2=3477196690058d+10027987161530$.

$\bullet$ $2\cdot 329289+1=658579$,\hspace*{22.1mm} $2\cdot 10027987161530^2=5153245080958d+10764345756918$.

$\bullet$ $2\cdot 658579+1=1317159$,\hspace*{20.8mm} $2\cdot 10764345756918^2=5937840629415d+7258376622663$.

$\bullet$ $2\cdot 1317159=2634318$,\hspace*{23.9mm} $7258376622663^2=1349902064087d+13688184444896$.

$\bullet$ $2\cdot 2634318=5268636$,\hspace*{23.9mm} $13688184444896^2=4800814885347d+11643098680603$.

$\bullet$ $2\cdot 5268636+1=10537273$,\hspace*{18.2mm} $2\cdot 11643098680603^2=6946889893478d+1481444325256$.

$\bullet$ $2\cdot 10537273+1=21074547$,\hspace*{16.9mm} $2\cdot 1481444325256^2=112466693794d+1299469525746$.

$\bullet$ $2\cdot 21074547+1=42149095$,\hspace*{16.9mm} $2\cdot 1299469525746^2=86533736574d+22516898956086$.

$\bullet$ $2\cdot 42149095+1=84298191$,\hspace*{16.9mm} $2\cdot 22516898956086^2=25981870673373d+12486243382125$.

$\bullet$ $2\cdot 84298191+1=168596383$,\hspace*{15.6mm} $2\cdot 12486243382125^2=7989449403329d+36029689713659$.

$\bullet$ $2\cdot 168596383=337192766$,\hspace*{18.9mm} $36029689713659^2=33261689661681d+18099097476082$.

$\bullet$ $2\cdot 337192766+1=674385533$,\hspace*{14.3mm} $2\cdot 18099097476082^2=16786768329189d+8643489516917$.

$\bullet$ $2\cdot 674385533=1348771066$,\hspace*{17.6mm} $8643489516917^2=1914262458958d+16355540998007$.

$\bullet$ $2\cdot 1348771066+1=2697542133$,\hspace*{11.8mm} $2\cdot 16355540998007^2=13708283796212d+16782573114550$.

$\bullet$ $2\cdot 2697542133+1=5395084267$,\hspace*{11.7mm} $2\cdot 16782573114550^2=14433456731225d+18430065073225$.

$\bullet$ $2\cdot 5395084267=10790168534$,\hspace*{15mm} $18430065073225^2=8703160655609d+10285064380914$.

}

\medskip
Observe that $\frac{d-1}{p}=10790168534=1010000011001001001110011111010110_2$. We infer by Lemma~\ref{Lemma 2.2} that $2^{\frac{d-1}{p}}-1\equiv 10285064380913\mod d$. Since $7336389398826\cdot 10285064380913=1933359658541d-1$, we have that ${\rm gcd}\left(2^{\frac{d-1}{p}}-1,d\right)=1$. Finally, we prove that $2^{\frac{d-1}{q}}\equiv 15901499388071\mod d$ by applying Lemma~\ref{Lemma 2.2}. We proceed along the lines of the proof of $2^{d-1}\equiv 1\mod d$ above.

\medskip

{\tiny
$\bullet$ $2\cdot 0+1=1$,\hspace*{35.1mm} $2\cdot 1^2=2$.

$\bullet$ $2\cdot 1=2$,\hspace*{39.5mm} $2^2=4$.

$\bullet$ $2\cdot 2=4$,\hspace*{39.5mm} $4^2=16$.

$\bullet$ $2\cdot 4+1=9$,\hspace*{35.1mm} $2\cdot 16^2=512$.

$\bullet$ $2\cdot 9=18$,\hspace*{38.2mm} $512^2=262144$.

$\bullet$ $2\cdot 18=36$,\hspace*{36.9mm} $262144^2=68719476736$.

$\bullet$ $2\cdot 36=72$,\hspace*{36.9mm} $68719476736^2=120999325d+36711407762021$.

$\bullet$ $2\cdot 72=144$,\hspace*{35.6mm} $36711407762021^2=34532286892056d+31998527837617$.

$\bullet$ $2\cdot 144+1=289$,\hspace*{29.9mm} $2\cdot 31998527837617^2=52470264691605d+38620486063583$.

$\bullet$ $2\cdot 289=578$,\hspace*{34.3mm} $38620486063583^2=38217188450990d+25764109643679$.

$\bullet$ $2\cdot 578+1=1157$,\hspace*{28.6mm} $2\cdot 25764109643679^2=34016022979769d+6589566597731$.

$\bullet$ $2\cdot 1157=2314$,\hspace*{31.7mm} $6589566597731^2=1112594647461d+6956899607542$.

$\bullet$ $2\cdot 2314=4628$,\hspace*{31.7mm} $6956899607542^2=1240094369611d+30110629581095$.

$\bullet$ $2\cdot 4628=9256$,\hspace*{31.7mm} $30110629581095^2=23230734194007d+5999347360672$.

$\bullet$ $2\cdot 9256=18512$,\hspace*{30.4mm} $5999347360672^2=922213084091d+7444474594995$.

$\bullet$ $2\cdot 18512+1=37025$,\hspace*{24.7mm} $2\cdot 7444474594995^2=2840019769443d+7552148495853$.

$\bullet$ $2\cdot 37025+1=74051$,\hspace*{24.7mm} $2\cdot 7552148495853^2=2922767707846d+10516484734984$.

$\bullet$ $2\cdot 74051=148102$,\hspace*{27.9mm} $10516484734984^2=2833769063220d+28051642057876$.

$\bullet$ $2\cdot 148102+1=296205$,\hspace*{22.1mm} $2\cdot 28051642057876^2=40324578454903d+38284942303215$.

$\bullet$ $2\cdot 296205=592410$,\hspace*{26.6mm} $38284942303215^2=37555993656179d+37268879953484$.

$\bullet$ $2\cdot 592410=1184820$,\hspace*{25.3mm} $37268879953484^2=35589013121550d+38577918665806$.

$\bullet$ $2\cdot 1184820+1=2369641$,\hspace*{19.5mm} $2\cdot 38577918665806^2=76265978220592d+10176084091704$.

$\bullet$ $2\cdot 2369641+1=4739283$,\hspace*{19.5mm} $2\cdot 10176084091704^2=5306578989668d+15426040080260$.

$\bullet$ $2\cdot 4739283+1=9478567$,\hspace*{19.5mm} $2\cdot 15426040080260^2=12194448661681d+9515608643001$.

$\bullet$ $2\cdot 9478567=18957134$,\hspace*{22.7mm} $9515608643001^2=2320044993389d+36985824109670$.

$\bullet$ $2\cdot 18957134+1=37914269$,\hspace*{16.9mm} $2\cdot 36985824109670^2=70100942785264d+9512534908344$.

$\bullet$ $2\cdot 37914269+1=75828539$,\hspace*{16.9mm} $2\cdot 9512534908344^2=4637092784516d+22708882969508$.

$\bullet$ $2\cdot 75828539+1=151657079$,\hspace*{15.6mm} $2\cdot 22708882969508^2=26426813704896d+7292632926944$.

$\bullet$ $2\cdot 151657079+1=303314159$,\hspace*{14.3mm} $2\cdot 7292632926944^2=2725348009752d+37465133263064$.

$\bullet$ $2\cdot 303314159+1=606628319$,\hspace*{14.3mm} $2\cdot 37465133263064^2=71929629325755d+32577343114547$.

$\bullet$ $2\cdot 606628319+1=1213256639$,\hspace*{13mm} $2\cdot 32577343114547^2=54385682479598d+31854972276976$.

$\bullet$ $2\cdot 1213256639+1=2426513279$,\hspace*{11.7mm} $2\cdot 31854972276976^2=52000524212466d+35754130095938$.

$\bullet$ $2\cdot 2426513279+1=4853026559$,\hspace*{11.7mm} $2\cdot 35754130095938^2=65509712116177d+10262965651905$.

$\bullet$ $2\cdot 4853026559=9706053118$,\hspace*{16.2mm} $10262965651905^2=2698789513526d+15901499388071$.

}

\medskip
It is now clear that $\frac{d-1}{q}=9706053118=1001000010100001101001110111111110_2$. Moreover, $2^{\frac{d-1}{q}}-1\equiv 15901499388070\mod d$ by Lemma~\ref{Lemma 2.2}. Since $11826010015564\cdot 15901499388070=4818363746001d+1$, it follows that ${\rm gcd}\left(2^{\frac{d-1}{q}}-1,d\right)=1$. This shows that $d\in\mathbb{P}$.\qed(Claim 1)

\medskip
Finally, we show that $d\mid y$. Set $r=57$ and $s=56$, and let $\left(a_i\right)_{i=1}^r,\left(b_i\right)_{i=1}^r,\left(c_i\right)_{i=1}^s,\left(d_i\right)_{i=1}^s,\left(e_i\right)_{i=1}^s,\left(f_i\right)_{i=1}^s$ be defined as follows.

{\tiny
\begin{table}[htbp]
\centering
\begin{tabular}{|*{15}{c|}}
\cline{1-7}\cline{9-15}
$i$ & $a_i$ & $b_i$ & $c_i$ & $d_i$ & $e_i$ & $f_i$ & \quad \quad & $i$ & $a_i$ & $b_i$ & $c_i$ & $d_i$ & $e_i$ & $f_i$\\
\cline{1-7}\cline{9-15}
$1$ & $-3033477$ & $3410579$ & $2$ & $84702667$ & $1$ & $19514019793739$ & \quad \quad & $30$ & $3843$ & $1909615$ & $271$ & $16934137$ & $49$ & $144014906218$\\
\cline{1-7}\cline{9-15}
$2$ & $-2271496$ & $3206358$ & $5$ & $146634276$ & $2$ & $7805607917495$ & \quad \quad & $31$ & $29419$ & $6660471$ & $277$ & $5275486$ & $9$ & $140895449774$\\
\cline{1-7}\cline{9-15}
$3$ & $-1787484$ & $82862$ & $23$ & $112959051$ & $7$ & $1696871286410$ & \quad \quad & $32$ & $37368$ & $3205421$ & $283$ & $6324658$ & $96$ & $137908267061$\\
\cline{1-7}\cline{9-15}
$4$ & $-1758518$ & $4049619$ & $29$ & $57435327$ & $14$ & $1345794468527$ & \quad \quad & $33$ & $38302$ & $2075306$ & $307$ & $8891039$ & $68$ & $127127164765$\\
\cline{1-7}\cline{9-15}
$5$ & $-1699348$ & $2129966$ & $37$ & $54701496$ & $5$ & $1054811880742$ & \quad \quad & $34$ & $80167$ & $4486315$ & $311$ & $8094883$ & $49$ & $125492088698$\\
\cline{1-7}\cline{9-15}
$6$ & $-1427442$ & $8014362$ & $47$ & $60589478$ & $3$ & $830383821010$ & \quad \quad & $35$ & $83134$ & $7103618$ & $337$ & $8014362$ & $90$ & $115810206467$\\
\cline{1-7}\cline{9-15}
$7$ & $-832653$ & $717212$ & $53$ & $37470722$ & $22$ & $736378105415$ & \quad \quad & $36$ & $102633$ & $6324658$ & $349$ & $21686733$ & $103$ & $111828193630$\\
\cline{1-7}\cline{9-15}
$8$ & $-748505$ & $9346906$ & $61$ & $9465381$ & $6$ & $639803927663$ & \quad \quad & $37$ & $115435$ & $9578599$ & $353$ & $17200418$ & $128$ & $110561018615$\\
\cline{1-7}\cline{9-15}
$9$ & $-674823$ & $15717595$ & $79$ & $40989357$ & $9$ & $494025817562$ & \quad \quad & $38$ & $116792$ & $11851993$ & $359$ & $6177547$ & $132$ & $108713202145$\\
\cline{1-7}\cline{9-15}
$10$ & $-554788$ & $3962534$ & $83$ & $26886549$ & $11$ & $470217344426$ & \quad \quad & $39$ & $156001$ & $4568510$ & $383$ & $35556425$ & $23$ & $101900886650$\\
\cline{1-7}\cline{9-15}
$11$ & $-521509$ & $17200418$ & $89$ & $16628268$ & $8$ & $438517298735$ & \quad \quad & $40$ & $197807$ & $3013477$ & $409$ & $11851993$ & $182$ & $95423079595$\\
\cline{1-7}\cline{9-15}
$12$ & $-482142$ & $13587629$ & $97$ & $33227623$ & $5$ & $402350923582$ & \quad \quad & $41$ & $211873$ & $4414221$ & $421$ & $4049619$ & $1$ & $92703181918$\\
\cline{1-7}\cline{9-15}
$13$ & $-445057$ & $8567013$ & $107$ & $27266087$ & $25$ & $364748033522$ & \quad \quad & $42$ & $219482$ & $1158859$ & $433$ & $1909615$ & $54$ & $90134040611$\\
\cline{1-7}\cline{9-15}
$14$ & $-425248$ & $2262519$ & $109$ & $16545406$ & $22$ & $358055409055$ & \quad \quad & $43$ & $229892$ & $6177547$ & $439$ & $3410579$ & $13$ & $88902140290$\\
\cline{1-7}\cline{9-15}
$15$ & $-322867$ & $3536609$ & $113$ & $31075882$ & $50$ & $345380881283$ & \quad \quad & $44$ & $264088$ & $6019577$ & $443$ & $13587629$ & $158$ & $88099412105$\\
\cline{1-7}\cline{9-15}
$16$ & $-280973$ & $3410579$ & $137$ & $14887445$ & $36$ & $284876201359$ & \quad \quad & $45$ & $269512$ & $147111$ & $449$ & $5020616$ & $202$ & $86922137075$\\
\cline{1-7}\cline{9-15}
$17$ & $-252892$ & $4558274$ & $139$ & $18716004$ & $49$ & $280777263202$ & \quad \quad & $46$ & $315998$ & $4315464$ & $467$ & $20082812$ & $8$ & $83571819245$\\
\cline{1-7}\cline{9-15}
$18$ & $-226498$ & $2007139$ & $157$ & $18190318$ & $57$ & $248586239390$ & \quad \quad & $47$ & $339243$ & $1212808$ & $557$ & $26469383$ & $149$ & $70068293654$\\
\cline{1-7}\cline{9-15}
$19$ & $-215202$ & $3644517$ & $163$ & $23253236$ & $3$ & $239435825690$ & \quad \quad & $48$ & $471496$ & $2548$ & $563$ & $15717595$ & $214$ & $69321562241$\\
\cline{1-7}\cline{9-15}
$20$ & $-207973$ & $4049619$ & $167$ & $4568510$ & $23$ & $233700835850$ & \quad \quad & $49$ & $499231$ & $1109106$ & $569$ & $5278544$ & $205$ & $68590579166$\\
\cline{1-7}\cline{9-15}
$21$ & $-206198$ & $5460335$ & $179$ & $16995512$ & $63$ & $218033740690$ & \quad \quad & $50$ & $596696$ & $1909615$ & $593$ & $6589566$ & $75$ & $65814569278$\\
\cline{1-7}\cline{9-15}
$22$ & $-193242$ & $2284255$ & $193$ & $16028724$ & $94$ & $202217821651$ & \quad \quad & $51$ & $1124049$ & $5020616$ & $601$ & $4414221$ & $280$ & $64938501679$\\
\cline{1-7}\cline{9-15}
$23$ & $-126128$ & $8014362$ & $199$ & $21271274$ & $15$ & $196120801946$ & \quad \quad & $52$ & $1154432$ & $1909615$ & $631$ & $13309440$ & $139$ & $61851092818$\\
\cline{1-7}\cline{9-15}
$24$ & $-121640$ & $5480460$ & $223$ & $1909615$ & $52$ & $175013630425$ & \quad \quad & $53$ & $1214853$ & $23621570$ & $641$ & $17283280$ & $265$ & $60886177094$\\
\cline{1-7}\cline{9-15}
$25$ & $-54422$ & $13915416$ & $227$ & $19110033$ & $90$ & $171929689777$ & \quad \quad & $54$ & $1403789$ & $16545406$ & $647$ & $13915416$ & $74$ & $60321544949$\\
\cline{1-7}\cline{9-15}
$26$ & $-48408$ & $8365791$ & $229$ & $9530079$ & $41$ & $170428120462$ & \quad \quad & $55$ & $1947578$ & $5128375$ & $683$ & $11066152$ & $192$ & $57142078405$\\
\cline{1-7}\cline{9-15}
$27$ & $-40763$ & $864323$ & $239$ & $4315464$ & $40$ & $163297236761$ & \quad \quad & $56$ & $1991126$ & $10145918$ & $691$ & $9578599$ & $38$ & $56480520385$\\
\cline{1-7}\cline{9-15}
$28$ & $-34641$ & $7529543$ & $263$ & $13737288$ & $63$ & $148395587770$ & \quad \quad & $57$ & $2126573$ & $6589566$\\
\cline{1-7}\cline{9-11}
$29$ & $-22980$ & $14887445$ & $269$ & $6324658$ & $125$ & $145085648966$\\
\cline{1-7}
\end{tabular}
\end{table}

}

For each $j\in [1,s]$, set $P_j=c_j\mathcal{O}_K+\left(e_j+\sqrt{d}\right)\mathcal{O}_K$ and $\overline{P}_j=c_j\mathcal{O}_K+\left(e_j-\sqrt{d}\right)\mathcal{O}_K$. Then $P_1=\overline{P}_1$, $c_1f_1+e_1^2=d$, $c_1\in\mathbb{P}$ and ${\rm gcd}\left(c_1,e_1\right)=1$. Observe that $c_jf_j+e_j^2=d$, $c_j\in\mathbb{P}$ and ${\rm gcd}\left(c_j,2e_j\right)=1$ for each $j\in [2,s]$. In particular, $P_j,\overline{P}_j\in\max\left(\mathcal{O}_K\right)$, $c_j\mathcal{O}_K=P_j\overline{P}_j$ and ${\rm N}\left(P_j\right)={\rm N}\left(\overline{P}_j\right)=c_j$ for each $j\in [1,s]$. Set $z=\prod_{i=1}^r\left(a_i+\sqrt{d}\right)^{b_i}$, set $n=\prod_{j=1}^s c_j^{d_j}$ and set $\eta=\frac{z}{n}$. Clearly, $\eta\in K$ and $\eta>0$.

\medskip
\textsc{Claim} 2: $\eta\in\mathcal{O}_K^{\times}$.

\medskip
Observe the following.

{\tiny
\begin{align*}
d=&1561121583499c_2^2+2^2=312224316694c_2^3+27^2=12488972039c_2^5+1402^2=73777012451c_3^2+30^2=46406705350c_4^2+623^2\\
=&17667740470c_6^2+943^2=13893926015c_7^2+1188^2=5665268302c_{10}^2+2501^2=3056458675c_{15}^2+10898^2=1399382135c_{20}^2+25908^2\\
=&1218065438c_{21}^2+2211^2=1047746695c_{22}^2+22868^2=320417518c_{36}^2+29419^2=125690710c_{47}^2+180617^2,
\end{align*}

$2=e_2$, $27=5c_2+e_2$, $1402=280c_2+e_2$, $30=c_3+e_3$, $623=21c_4+e_4$, $943=20c_6+e_6$, $1188=22c_7+e_7$, $2501=30c_{10}+e_{10}$, $10898=96c_{15}+e_{15}$, $25908=155c_{20}+e_{20}$, $2211=12c_{21}+e_{21}$, $22868=118c_{22}+e_{22}$, $29419=84c_{36}+e_{36}$, $180617=324c_{47}+e_{47}$.

}

\medskip
Let $j\in [2,s]$, $k,m\in\mathbb{N}_0$, $\ell\in\mathbb{N}$ be such that $d=kc_j^{\ell}+\left(mc_j+e_j\right)^2$. Set $I=c_j^{\ell}\mathcal{O}_K+\left(mc_j+e_j+\sqrt{d}\right)\mathcal{O}_K$ and $\overline{I}=c_j^{\ell}\mathcal{O}_K+\left(mc_j+e_j-\sqrt{d}\right)\mathcal{O}_K$. Clearly, $I$ and $\overline{I}$ are ideals of $\mathcal{O}_K$ such that $I\subseteq P_j$ and $\overline{I}\subseteq\overline{P}_j$. Observe that $I\overline{I}=c_j^{\ell}J$ for some ideal $J$ of $\mathcal{O}_K$ with $c_j^{\ell},2\left(mc_j+e_j\right)\in J$. Since ${\rm gcd}\left(c_j^{\ell},2\left(mc_j+e_j\right)\right)=1$, we obtain that $J=\mathcal{O}_K$, $I\nsubseteq\overline{P}_j$ and $\overline{I}\nsubseteq P_j$. Since $I\overline{I}=c_j^{\ell}\mathcal{O}_K=P_j^{\ell}\overline{P}_j^{\ell}$, we infer that $P_j^{\ell}=I$ and $\overline{P}_j^{\ell}=\overline{I}$. Using this, we get the following equations (and in analogy the equations for the conjugates $\overline{P}_j^{\ell}$).

\medskip
\textbf{Equations, part A}
\medskip

{\tiny
$P_2^2=c_2^2\mathcal{O}_K+\left(2+\sqrt{d}\right)\mathcal{O}_K$, $P_2^3=c_2^3\mathcal{O}_K+\left(27+\sqrt{d}\right)\mathcal{O}_K$, $P_2^5=c_2^5\mathcal{O}_K+\left(1402+\sqrt{d}\right)\mathcal{O}_K$, $P_3^2=c_3^2\mathcal{O}_K+\left(30+\sqrt{d}\right)\mathcal{O}_K$,

$P_4^2=c_4^2\mathcal{O}_K+\left(623+\sqrt{d}\right)\mathcal{O}_K$, $P_6^2=c_6^2\mathcal{O}_K+\left(943+\sqrt{d}\right)\mathcal{O}_K$, $P_7^2=c_7^2\mathcal{O}_K+\left(1188+\sqrt{d}\right)\mathcal{O}_K$, $P_{10}^2=c_{10}^2\mathcal{O}_K+\left(2501+\sqrt{d}\right)\mathcal{O}_K$,

$P_{15}^2=c_{15}^2\mathcal{O}_K+\left(10898+\sqrt{d}\right)\mathcal{O}_K$, $P_{20}^2=c_{20}^2\mathcal{O}_K+\left(25908+\sqrt{d}\right)\mathcal{O}_K$, $P_{21}^2=c_{21}^2\mathcal{O}_K+\left(2211+\sqrt{d}\right)\mathcal{O}_K$, $P_{22}^2=c_{22}^2\mathcal{O}_K+\left(22868+\sqrt{d}\right)\mathcal{O}_K$,

$P_{36}^2=c_{36}^2\mathcal{O}_K+\left(29419+\sqrt{d}\right)\mathcal{O}_K$, $P_{47}^2=c_{47}^2\mathcal{O}_K+\left(180617+\sqrt{d}\right)\mathcal{O}_K$.

}

\medskip
The following equations can easily be derived (by using elementary arithmetic).

\medskip
\textbf{Equations, part B}
\medskip

{\tiny
$a_1=-1516739c_1+e_1=-121339c_2^2-2=-131890c_3-e_3=-81986c_5+e_5=-64542c_6-e_6=-57235c_7-e_7=-6910c_{43}+e_{43}$

\hspace*{2.7mm} $=-4732c_{53}-e_{53}$,

$a_2=-4294c_3^2+30=-42858c_7-e_7=-28753c_9-e_9=-16580c_{16}-e_{16}=-9504c_{27}-e_{27}=-4864c_{46}-e_{46}$,

$a_3=-77717c_3+e_3=-20084c_{11}-e_{11}=-15818c_{15}-e_{15}=-4667c_{39}-e_{39}=-2833c_{52}+e_{52}=-2789c_{53}+e_{53}$,

$a_4=-351704c_2+e_2=-76457c_3-e_3=-33180c_7+e_7=-18129c_{12}-e_{12}=-4982c_{37}+e_{37}=-4300c_{40}+e_{40}=-4177c_{41}-e_{41}$,

$a_5=-13595c_2^3+27=-73885c_3+e_3=-17519c_{12}-e_{12}=-4437c_{39}+e_{39}=-3018c_{48}-e_{48}=-2828c_{51}+e_{51}$,

$a_6=-285488c_2-e_2=-62063c_3+e_3=-18069c_9+e_9=-13096c_{14}+e_{14}=-7173c_{23}-e_{23}=-4236c_{35}+e_{35}=-2563c_{47}+e_{47}$,

$a_7=-416327c_1+e_1=-166531c_2+e_2=-36202c_3-e_3=-22504c_5-e_5=-8584c_{12}-e_{12}=-3006c_{31}+e_{31}=-2319c_{38}-e_{38}$

\hspace*{2.7mm} $=-1783c_{46}+e_{46}$,

$a_8=-374253c_1+e_1=-1415c_3^2+30=-25811c_4+e_4=-20230c_5+e_5=-9018c_{10}-e_{10}=-1186c_{52}-e_{52}=-1157c_{54}+e_{54}$,

$a_9=-337412c_1+e_1=-26993c_2^2+2=-14358c_6+e_6=-4140c_{19}-e_{19}=-1762c_{39}+e_{39}=-1445c_{46}-e_{46}=-1199c_{48}+e_{48}$,

$a_{10}=-110958c_2+e_2=-5090c_{14}+e_{14}=-1572c_{37}+e_{37}=-1188c_{46}+e_{46}=-879c_{52}-e_{52}=-812c_{55}-e_{55}$,

$a_{11}=-260755c_1+e_1=-22674c_3-e_3=-11096c_6+e_6=-2297c_{25}-e_{25}=-1494c_{36}-e_{36}=-1477c_{37}-e_{37}=-814c_{53}+e_{53}$,

$a_{12}=-96428c_2-e_2=-20963c_3+e_3=-13031c_5+e_5=-3469c_{17}+e_{17}=-1833c_{28}-e_{28}=-1088c_{44}-e_{44}=-856c_{48}-e_{48}$,

$a_{13}=-222529c_1+e_1=-89011c_2-e_2=-19350c_3-e_3=-5362c_{10}-e_{10}=-3939c_{15}+e_{15}=-2486c_{21}-e_{21}=-1961c_{25}+e_{25}$

\hspace*{3.9mm} $=-1005c_{44}+e_{44}$,

$a_{14}=-17010c_2^2+2=-5383c_9+e_9=-2137c_{23}+e_{23}=-1569c_{30}-e_{30}=-747c_{49}-e_{49}=-663c_{53}-e_{53}$,

$a_{15}=-161434c_1+e_1=-64573c_2-e_2=-14038c_3+e_3=-8726c_5-e_5=-5293c_8+e_8=-1038c_{34}-e_{34}=-915c_{37}+e_{37}$

\hspace*{3.9mm} $=-473c_{55}+e_{55}$,

$a_{16}=-140487c_1+e_1=-2248c_2^3+27=-7594c_5+e_5=-1790c_{18}+e_{18}=-1412c_{23}+e_{23}=-915c_{33}-e_{33}=-640c_{43}-e_{43}$,

$a_{17}=-50578c_2-e_2=-478c_3^2-30=-933c_{30}-e_{30}=-913c_{31}+e_{31}=-813c_{34}-e_{34}=-401c_{52}+e_{52}$,

$a_{18}=-9060c_2^2+2=-7c_{21}^2-2211=-738c_{33}+e_{33}=-642c_{37}+e_{37}=-504c_{45}-e_{45}$,

$a_{19}=-8608c_2^2-2=-4060c_7-e_7=-3528c_8+e_8=-2011c_{13}-e_{13}=-16c_{15}^2-10898=-610c_{37}+e_{37}$,

$a_{20}=-103987c_1+e_1=-67c_2^5+1402=-9042c_3-e_3=-7171c_4-e_4=-2144c_{12}-e_{12}=-908c_{26}-e_{26}=-494c_{41}+e_{41}$,

$a_{21}=-8248c_2^2+2=-1313c_{18}-e_{18}=-1265c_{19}-e_{19}=-767c_{29}+e_{29}=-574c_{38}-e_{38}=-327c_{52}+e_{52}$,

$a_{22}=-38648c_2-e_2=-6664c_4+e_4=-3168c_8+e_8=-6c_{20}^2-25908=-735c_{28}+e_{28}=-322c_{51}+e_{51}$,

$a_{23}=-25226c_2+e_2=-3409c_5+e_5=-1179c_{13}+e_{13}=-803c_{18}-e_{18}=-4c_{22}^2+22868=-374c_{35}-e_{35}$,

$a_{24}=-230c_3^2+30=-4194c_4-e_4=-1994c_8-e_8=-531c_{26}-e_{26}=-396c_{33}-e_{33}=-205c_{50}-e_{50}$,

$a_{25}=-10884c_2-e_2=-1471c_5+e_5=-689c_9+e_9=-561c_{12}-e_{12}=-347c_{18}+e_{18}=-201c_{30}+e_{30}=-84c_{54}-e_{54}$,

$a_{26}=-9682c_2+e_2=-2105c_3+e_3=-544c_{11}+e_{11}=-499c_{12}-e_{12}=-297c_{19}+e_{19}=-139c_{36}+e_{36}=-70c_{56}-e_{56}$,

$a_{27}=-20382c_1+e_1=-8153c_2+e_2=-77c_3^2-30=-260c_{18}+e_{18}=-152c_{29}+e_{29}=-133c_{33}+e_{33}=-72c_{49}+e_{49}$,

$a_{28}=-17321c_1+e_1=-1195c_4+e_4=-307c_{15}+e_{15}=-179c_{22}-e_{22}=-174c_{23}-e_{23}=-153c_{25}+e_{25}=-51c_{55}+e_{55}$,

$a_{29}=-489c_6+e_6=-434c_7+e_7=-291c_9+e_9=-277c_{10}+e_{10}=-215c_{13}+e_{13}=-168c_{16}+e_{16}=-141c_{19}+e_{19}$,

$a_{30}=1921c_1+e_1=769c_2-e_2=133c_4-e_4=104c_5-e_5=28c_{17}-e_{17}=17c_{24}+e_{24}=14c_{30}+e_{30}=9c_{42}-e_{42}$,

$a_{31}=14709c_1+e_1=626c_6-e_6=215c_{16}-e_{16}=212c_{17}-e_{17}=164c_{21}+e_{21}=29419$,

$a_{32}=7474c_2-e_2=1625c_3-e_3=795c_6+e_6=349c_{13}+e_{13}=163c_{26}+e_{26}=80c_{46}+e_{46}=59c_{52}+e_{52}$,

$a_{33}=1532c_2^2+2=1665c_3+e_3=815c_6-e_6=628c_8-e_8=235c_{19}-e_{19}=123c_{34}+e_{34}=82c_{46}+e_{46}$,

$a_{34}=40083c_1+e_1=16033c_2+e_2=1513c_7-e_7=12c_{10}^2-2501=709c_{15}+e_{15}=296c_{30}-e_{30}=230c_{36}-e_{36}$,

$a_{35}=2247c_5-e_5=934c_{11}+e_{11}=857c_{12}+e_{12}=217c_{39}+e_{39}=178c_{46}+e_{46}=122c_{55}-e_{55}$,

$a_{36}=51316c_1+e_1=20527c_2-e_2=4462c_3+e_3=2774c_5-e_5=448c_{26}+e_{26}=390c_{28}+e_{28}=382c_{29}-e_{29}=363c_{32}-e_{32}$,

$a_{37}=57717c_1+e_1=138c_4^2-623=3120c_5-e_5=2456c_6+e_6=1190c_{12}+e_{12}=580c_{23}+e_{23}=167c_{56}+e_{56}$,

$a_{38}=23358c_2+e_2=2485c_6-e_6=1407c_{10}+e_{10}=10c_{15}^2-10898=305c_{39}-e_{39}=286c_{40}-e_{40}$,

$a_{39}=78000c_1+e_1=2943c_7+e_7=1431c_{14}+e_{14}=994c_{18}-e_{18}=934c_{20}+e_{20}=784c_{23}-e_{23}=241c_{54}+e_{54}$,

$a_{40}=98903c_1+e_1=39561c_2+e_2=8600c_3+e_3=5346c_5+e_5=2504c_9-e_9=871c_{25}+e_{25}=441c_{45}-e_{45}=348c_{49}-e_{49}$,

$a_{41}=105936c_1+e_1=8475c_2^2-2=4508c_6-e_6=1184c_{21}-e_{21}=782c_{30}-e_{30}=372c_{49}+e_{49}=353c_{51}-e_{51}$,

$a_{42}=43896c_2+e_2=9543c_3-e_3=2466c_{11}+e_{11}=2051c_{13}+e_{13}=1103c_{23}-e_{23}=573c_{39}+e_{39}=470c_{46}-e_{46}$,

$a_{43}=45978c_2+e_2=9995c_3+e_3=4338c_7-e_7=2034c_{15}+e_{15}=812c_{32}+e_{32}=640c_{38}+e_{38}=413c_{47}-e_{47}$,

$a_{44}=52818c_2-e_2=9106c_4+e_4=3343c_9-e_9=1475c_{21}+e_{21}=1327c_{23}+e_{23}=860c_{33}+e_{33}=849c_{34}+e_{34}$,

$a_{45}=53902c_2+e_2=9294c_4-e_4=3247c_{10}+e_{10}=1717c_{18}-e_{18}=1025c_{28}-e_{28}=973c_{31}-e_{31}=952c_{32}+e_{32}$,

$a_{46}=12640c_2^2-2=375c_4^2+623=2796c_{15}+e_{15}=1765c_{21}+e_{21}=1322c_{27}+e_{27}=825c_{39}+e_{39}$,

$a_{47}=169621c_1+e_1=67849c_2-e_2=14750c_3-e_3=154c_6^2-943=1252c_{30}-e_{30}=829c_{40}+e_{40}=491c_{56}-e_{56}$,

$a_{48}=16258c_4+e_4=12743c_5+e_5=1793c_{28}-e_{28}=1231c_{39}+e_{39}=829c_{49}-e_{49}=747c_{52}+e_{52}$,

$a_{49}=249615c_1+e_1=21706c_3-e_3=10622c_6-e_6=2089c_{27}-e_{27}=1842c_{30}+e_{30}=1069c_{46}+e_{46}=842c_{50}-e_{50}$,

$a_{50}=25943c_3+e_3=212c_7^2+1188=9782c_8-e_8=7553c_9+e_9=2676c_{24}-e_{24}=1071c_{47}+e_{47}$,

$a_{51}=562024c_1+e_1=48872c_3-e_3=23916c_6-e_6=8205c_{16}-e_{16}=2537c_{44}+e_{44}=2503c_{45}+e_{45}=1754c_{53}-e_{53}$,

$a_{52}=230886c_2+e_2=50193c_3-e_3=31201c_5-e_5=5982c_{22}-e_{22}=5086c_{25}-e_{25}=2666c_{42}+e_{42}=2472c_{46}+e_{46}$,

$a_{53}=607426c_1+e_1=242971c_2-e_2=52820c_3-e_3=41891c_4+e_4=32834c_5-e_5=25848c_6-e_6=15378c_9-e_9=11354c_{13}-e_{13}$

\hspace*{3.9mm} $=3172c_{39}-e_{39}$,

$a_{54}=701894c_1+e_1=61034c_3+e_3=48407c_4-e_4=26487c_7-e_7=15773c_{11}-e_{11}=14472c_{12}+e_{12}=12879c_{14}-e_{14}=2520c_{47}+e_{47}$,

$a_{55}=389516c_2-e_2=84677c_3+e_3=24653c_9-e_9=14011c_{17}+e_{17}=7405c_{28}+e_{28}=7031c_{31}-e_{31}=5085c_{39}+e_{39}$,

$a_{56}=3764c_3^2-30=68660c_4-e_4=37568c_7+e_7=14325c_{17}-e_{17}=7c_{47}^2-180617$,

$a_{57}=1063286c_1+e_1=85063c_2^2-2=92460c_3-e_3=11019c_{22}-e_{22}=5199c_{40}+e_{40}=3586c_{50}+e_{50}=3318c_{53}-e_{53}$.

}

\medskip
It is straightforward to prove that the equations below are satisfied. Note that these equations will be used later to show that the ideal norms of certain nonzero ideals of $\mathcal{O}_K$ coincide.

\medskip
\textbf{Equations, part C}
\medskip

{\tiny
\begin{align*}
d=&a_1^2+c_1c_2^2c_3c_5c_6c_7c_{43}c_{53}=a_2^2+c_3^2c_7c_9c_{16}c_{27}c_{46}=a_3^2+c_3c_{11}c_{15}c_{39}c_{52}c_{53}=a_4^2+c_2c_3c_7c_{12}c_{37}c_{40}c_{41}\\
=&a_5^2+c_2^3c_3c_{12}c_{39}c_{48}c_{51}=a_6^2+c_2c_3c_9c_{14}c_{23}c_{35}c_{47}=a_7^2+c_1c_2c_3c_5c_{12}c_{31}c_{38}c_{46}=a_8^2+c_1c_3^2c_4c_5c_{10}c_{52}c_{54}\\
=&a_9^2+c_1c_2^2c_6c_{19}c_{39}c_{46}c_{48}=a_{10}^2+c_2c_{14}c_{37}c_{46}c_{52}c_{55}=a_{11}^2+c_1c_3c_6c_{25}c_{36}c_{37}c_{53}=a_{12}^2+c_2c_3c_5c_{17}c_{28}c_{44}c_{48}\\
=&a_{13}^2+c_1c_2c_3c_{10}c_{15}c_{21}c_{25}c_{44}=a_{14}^2+c_2^2c_9c_{23}c_{30}c_{49}c_{53}=a_{15}^2+c_1c_2c_3c_5c_8c_{34}c_{37}c_{55}=a_{16}^2+c_1c_2^3c_5c_{18}c_{23}c_{33}c_{43}\\
=&a_{17}^2+c_2c_3^2c_{30}c_{31}c_{34}c_{52}=a_{18}^2+c_2^2c_{21}^2c_{33}c_{37}c_{45}=a_{19}^2+c_2^2c_7c_8c_{13}c_{15}^2c_{37}=a_{20}^2+c_1c_2^5c_3c_4c_{12}c_{26}c_{41}\\
=&a_{21}^2+c_2^2c_{18}c_{19}c_{29}c_{38}c_{52}=a_{22}^2+c_2c_4c_8c_{20}^2c_{28}c_{51}=a_{23}^2+c_2c_5c_{13}c_{18}c_{22}^2c_{35}=a_{24}^2+c_3^2c_4c_8c_{26}c_{33}c_{50}\\
=&a_{25}^2+c_2c_5c_9c_{12}c_{18}c_{30}c_{54}=a_{26}^2+c_2c_3c_{11}c_{12}c_{19}c_{36}c_{56}=a_{27}^2+c_1c_2c_3^2c_{18}c_{29}c_{33}c_{49}=a_{28}^2+c_1c_4c_{15}c_{22}c_{23}c_{25}c_{55}\\
=&a_{29}^2+c_6c_7c_9c_{10}c_{13}c_{16}c_{19}=a_{30}^2+c_1c_2c_4c_5c_{17}c_{24}c_{30}c_{42}=a_{31}^2+c_1c_6c_{16}c_{17}c_{21}c_{36}^2=a_{32}^2+c_2c_3c_6c_{13}c_{26}c_{46}c_{52}\\
=&a_{33}^2+c_2^2c_3c_6c_8c_{19}c_{34}c_{46}=a_{34}^2+c_1c_2c_7c_{10}^2c_{15}c_{30}c_{36}=a_{35}^2+c_5c_{11}c_{12}c_{39}c_{46}c_{55}=a_{36}^2+c_1c_2c_3c_5c_{26}c_{28}c_{29}c_{32}\\
=&a_{37}^2+c_1c_4^2c_5c_6c_{12}c_{23}c_{56}=a_{38}^2+c_2c_6c_{10}c_{15}^2c_{39}c_{40}=a_{39}^2+c_1c_7c_{14}c_{18}c_{20}c_{23}c_{54}=a_{40}^2+c_1c_2c_3c_5c_9c_{25}c_{45}c_{49}\\
=&a_{41}^2+c_1c_2^2c_6c_{21}c_{30}c_{49}c_{51}=a_{42}^2+c_2c_3c_{11}c_{13}c_{23}c_{39}c_{46}=a_{43}^2+c_2c_3c_7c_{15}c_{32}c_{38}c_{47}=a_{44}^2+c_2c_4c_9c_{21}c_{23}c_{33}c_{34}\\
=&a_{45}^2+c_2c_4c_{10}c_{18}c_{28}c_{31}c_{32}=a_{46}^2+c_2^2c_4^2c_{15}c_{21}c_{27}c_{39}=a_{47}^2+c_1c_2c_3c_6^2c_{30}c_{40}c_{56}=a_{48}^2+c_4c_5c_{28}c_{39}c_{49}c_{52}\\
=&a_{49}^2+c_1c_3c_6c_{27}c_{30}c_{46}c_{50}=a_{50}^2+c_3c_7^2c_8c_9c_{24}c_{47}=a_{51}^2+c_1c_3c_6c_{16}c_{44}c_{45}c_{53}=a_{52}^2+c_2c_3c_5c_{22}c_{25}c_{42}c_{46}\\
=&a_{53}^2+c_1c_2c_3c_4c_5c_6c_9c_{13}c_{39}=a_{54}^2+c_1c_3c_4c_7c_{11}c_{12}c_{14}c_{47}=a_{55}^2+c_2c_3c_9c_{17}c_{28}c_{31}c_{39}=a_{56}^2+c_3^2c_4c_7c_{17}c_{47}^2\\
=&a_{57}^2+c_1c_2^2c_3c_{22}c_{40}c_{50}c_{53}.
\end{align*}
}

\medskip
Observe that $2$ and $d$ are precisely the ramified primes of $\mathcal{O}_K$ (since $d\in\mathbb{P}$ and $4d$ is the discriminant of $\mathcal{O}_K$). Therefore, $c_j$ splits in $\mathcal{O}_K$ for each $j\in [2,s]$. We infer that all positive powers of the maximal ideals $P_1$, $P_j$ and $\overline{P}_j$ are pairwise comaximal for each $j\in [2,s]$. Next, we discuss an elementary fact that will subsequently be used. Let $I$ and $J$ be nonzero ideals of $\mathcal{O}_K$ such that $I\subseteq J$ and ${\rm N}\left(I\right)={\rm N}\left(J\right)$. It follows by the theorem of Lagrange that $|\mathcal{O}_K/J|={\rm N}\left(J\right)={\rm N}\left(I\right)=|\mathcal{O}_K/I|=|\mathcal{O}_K/J||J/I|$, and thus $|J/I|=1$. This implies that $I=J$.

\medskip
By combining the equations in the parts A, B and C above, we obtain the next list of equations. Note that the equations in parts A and B together show that each left hand side is contained in each right hand side. The equations in part C prove that the ideal norms on each side coincide, and hence both sides coincide. We will make these statements clearer, by giving a detailed proof for the first equality below. (The remaining equations can be proved in analogy.)

\medskip
\textbf{Equations, part D}
\medskip

{\tiny
$\left(a_1+\sqrt{d}\right)\mathcal{O}_K=P_1\overline{P}_2^2\overline{P}_3P_5\overline{P}_6\overline{P}_7P_{43}\overline{P}_{53}$, $\left(a_2+\sqrt{d}\right)\mathcal{O}_K=P_3^2\overline{P}_7\overline{P}_9\overline{P}_{16}\overline{P}_{27}\overline{P}_{46}$,
$\left(a_3+\sqrt{d}\right)\mathcal{O}_K=P_3\overline{P}_{11}\overline{P}_{15}\overline{P}_{39}P_{52}P_{53}$,

$\left(a_4+\sqrt{d}\right)\mathcal{O}_K=P_2\overline{P}_3P_7\overline{P}_{12}P_{37}P_{40}\overline{P}_{41}$,
$\left(a_5+\sqrt{d}\right)\mathcal{O}_K=P_2^3P_3\overline{P}_{12}P_{39}\overline{P}_{48}P_{51}$, $\left(a_6+\sqrt{d}\right)\mathcal{O}_K=\overline{P}_2P_3P_9P_{14}\overline{P}_{23}P_{35}P_{47}$,

$\left(a_7+\sqrt{d}\right)\mathcal{O}_K=P_1P_2\overline{P}_3\overline{P}_5\overline{P}_{12}P_{31}\overline{P}_{38}P_{46}$, $\left(a_8+\sqrt{d}\right)\mathcal{O}_K=P_1P_3^2P_4P_5\overline{P}_{10}\overline{P}_{52}P_{54}$,
$\left(a_9+\sqrt{d}\right)\mathcal{O}_K=P_1P_2^2P_6\overline{P}_{19}P_{39}\overline{P}_{46}P_{48}$,

$\left(a_{10}+\sqrt{d}\right)\mathcal{O}_K=P_2P_{14}P_{37}P_{46}\overline{P}_{52}\overline{P}_{55}$,
$\left(a_{11}+\sqrt{d}\right)\mathcal{O}_K=P_1\overline{P}_3P_6\overline{P}_{25}\overline{P}_{36}\overline{P}_{37}P_{53}$, $\left(a_{12}+\sqrt{d}\right)\mathcal{O}_K=\overline{P}_2P_3P_5P_{17}\overline{P}_{28}\overline{P}_{44}\overline{P}_{48}$,

$\left(a_{13}+\sqrt{d}\right)\mathcal{O}_K=P_1\overline{P}_2\overline{P}_3\overline{P}_{10}P_{15}\overline{P}_{21}P_{25}P_{44}$, $\left(a_{14}+\sqrt{d}\right)\mathcal{O}_K=P_2^2P_9P_{23}\overline{P}_{30}\overline{P}_{49}\overline{P}_{53}$,
$\left(a_{15}+\sqrt{d}\right)\mathcal{O}_K=P_1\overline{P}_2P_3\overline{P}_5P_8\overline{P}_{34}P_{37}P_{55}$,

$\left(a_{16}+\sqrt{d}\right)\mathcal{O}_K=P_1P_2^3P_5P_{18}P_{23}\overline{P}_{33}\overline{P}_{43}$,
$\left(a_{17}+\sqrt{d}\right)\mathcal{O}_K=\overline{P}_2\overline{P}_3^2\overline{P}_{30}P_{31}\overline{P}_{34}P_{52}$, $\left(a_{18}+\sqrt{d}\right)\mathcal{O}_K=P_2^2\overline{P}_{21}^2P_{33}P_{37}\overline{P}_{45}$,

$\left(a_{19}+\sqrt{d}\right)\mathcal{O}_K=\overline{P}_2^2\overline{P}_7P_8\overline{P}_{13}\overline{P}_{15}^2P_{37}$, $\left(a_{20}+\sqrt{d}\right)\mathcal{O}_K=P_1P_2^5\overline{P}_3\overline{P}_4\overline{P}_{12}\overline{P}_{26}P_{41}$,
$\left(a_{21}+\sqrt{d}\right)\mathcal{O}_K=P_2^2\overline{P}_{18}\overline{P}_{19}P_{29}\overline{P}_{38}P_{52}$,

$\left(a_{22}+\sqrt{d}\right)\mathcal{O}_K=\overline{P}_2P_4P_8\overline{P}_{20}^2P_{28}P_{51}$,
$\left(a_{23}+\sqrt{d}\right)\mathcal{O}_K=P_2P_5P_{13}\overline{P}_{18}P_{22}^2\overline{P}_{35}$, $\left(a_{24}+\sqrt{d}\right)\mathcal{O}_K=P_3^2\overline{P}_4\overline{P}_8\overline{P}_{26}\overline{P}_{33}\overline{P}_{50}$,

$\left(a_{25}+\sqrt{d}\right)\mathcal{O}_K=\overline{P}_2P_5P_9\overline{P}_{12}P_{18}P_{30}\overline{P}_{54}$, $\left(a_{26}+\sqrt{d}\right)\mathcal{O}_K=P_2P_3P_{11}\overline{P}_{12}P_{19}P_{36}\overline{P}_{56}$,
$\left(a_{27}+\sqrt{d}\right)\mathcal{O}_K=P_1P_2\overline{P}_3^2P_{18}P_{29}P_{33}P_{49}$,

$\left(a_{28}+\sqrt{d}\right)\mathcal{O}_K=P_1P_4P_{15}\overline{P}_{22}\overline{P}_{23}P_{25}P_{55}$,
$\left(a_{29}+\sqrt{d}\right)\mathcal{O}_K=P_6P_7P_9P_{10}P_{13}P_{16}P_{19}$, $\left(a_{30}+\sqrt{d}\right)\mathcal{O}_K=P_1\overline{P}_2\overline{P}_4\overline{P}_5\overline{P}_{17}P_{24}P_{30}\overline{P}_{42}$,

$\left(a_{31}+\sqrt{d}\right)\mathcal{O}_K=P_1\overline{P}_6\overline{P}_{16}\overline{P}_{17}P_{21}P_{36}^2$, $\left(a_{32}+\sqrt{d}\right)\mathcal{O}_K=\overline{P}_2\overline{P}_3P_6P_{13}P_{26}P_{46}P_{52}$,
$\left(a_{33}+\sqrt{d}\right)\mathcal{O}_K=P_2^2P_3\overline{P}_6\overline{P}_8\overline{P}_{19}P_{34}P_{46}$,

$\left(a_{34}+\sqrt{d}\right)\mathcal{O}_K=P_1P_2\overline{P}_7\overline{P}_{10}^2P_{15}\overline{P}_{30}\overline{P}_{36}$,
$\left(a_{35}+\sqrt{d}\right)\mathcal{O}_K=\overline{P}_5P_{11}P_{12}P_{39}P_{46}\overline{P}_{55}$, $\left(a_{36}+\sqrt{d}\right)\mathcal{O}_K=P_1\overline{P}_2P_3\overline{P}_5P_{26}P_{28}\overline{P}_{29}\overline{P}_{32}$,

$\left(a_{37}+\sqrt{d}\right)\mathcal{O}_K=P_1\overline{P}_4^2\overline{P}_5P_6P_{12}P_{23}P_{56}$, $\left(a_{38}+\sqrt{d}\right)\mathcal{O}_K=P_2\overline{P}_6P_{10}\overline{P}_{15}^2\overline{P}_{39}\overline{P}_{40}$,
$\left(a_{39}+\sqrt{d}\right)\mathcal{O}_K=P_1P_7P_{14}\overline{P}_{18}P_{20}\overline{P}_{23}P_{54}$,

$\left(a_{40}+\sqrt{d}\right)\mathcal{O}_K=P_1P_2P_3P_5\overline{P}_9P_{25}\overline{P}_{45}\overline{P}_{49}$,
$\left(a_{41}+\sqrt{d}\right)\mathcal{O}_K=P_1\overline{P}_2^2\overline{P}_6\overline{P}_{21}\overline{P}_{30}P_{49}\overline{P}_{51}$, $\left(a_{42}+\sqrt{d}\right)\mathcal{O}_K=P_2\overline{P}_3P_{11}P_{13}\overline{P}_{23}P_{39}\overline{P}_{46}$,

$\left(a_{43}+\sqrt{d}\right)\mathcal{O}_K=P_2P_3\overline{P}_7P_{15}P_{32}P_{38}\overline{P}_{47}$, $\left(a_{44}+\sqrt{d}\right)\mathcal{O}_K=\overline{P}_2P_4\overline{P}_9P_{21}P_{23}P_{33}P_{34}$,
$\left(a_{45}+\sqrt{d}\right)\mathcal{O}_K=P_2\overline{P}_4P_{10}\overline{P}_{18}\overline{P}_{28}\overline{P}_{31}P_{32}$,

$\left(a_{46}+\sqrt{d}\right)\mathcal{O}_K=\overline{P}_2^2P_4^2P_{15}P_{21}P_{27}P_{39}$,
$\left(a_{47}+\sqrt{d}\right)\mathcal{O}_K=P_1\overline{P}_2\overline{P}_3\overline{P}_6^2\overline{P}_{30}P_{40}\overline{P}_{56}$, $\left(a_{48}+\sqrt{d}\right)\mathcal{O}_K=P_4P_5\overline{P}_{28}P_{39}\overline{P}_{49}P_{52}$,

$\left(a_{49}+\sqrt{d}\right)\mathcal{O}_K=P_1\overline{P}_3\overline{P}_6\overline{P}_{27}P_{30}P_{46}\overline{P}_{50}$, $\left(a_{50}+\sqrt{d}\right)\mathcal{O}_K=P_3P_7^2\overline{P}_8P_9\overline{P}_{24}P_{47}$,
$\left(a_{51}+\sqrt{d}\right)\mathcal{O}_K=P_1\overline{P}_3\overline{P}_6\overline{P}_{16}P_{44}P_{45}\overline{P}_{53}$,

$\left(a_{52}+\sqrt{d}\right)\mathcal{O}_K=P_2\overline{P}_3\overline{P}_5\overline{P}_{22}\overline{P}_{25}P_{42}P_{46}$,
$\left(a_{53}+\sqrt{d}\right)\mathcal{O}_K=P_1\overline{P}_2\overline{P}_3P_4\overline{P}_5\overline{P}_6\overline{P}_9\overline{P}_{13}\overline{P}_{39}$, $\left(a_{54}+\sqrt{d}\right)\mathcal{O}_K=P_1P_3\overline{P}_4\overline{P}_7\overline{P}_{11}P_{12}\overline{P}_{14}P_{47}$,

$\left(a_{55}+\sqrt{d}\right)\mathcal{O}_K=\overline{P}_2P_3\overline{P}_9P_{17}P_{28}\overline{P}_{31}P_{39}$, $\left(a_{56}+\sqrt{d}\right)\mathcal{O}_K=\overline{P}_3^2\overline{P}_4P_7\overline{P}_{17}\overline{P}_{47}^2$,
$\left(a_{57}+\sqrt{d}\right)\mathcal{O}_K=P_1\overline{P}_2^2\overline{P}_3\overline{P}_{22}P_{40}P_{50}\overline{P}_{53}$.

}

\medskip
Now, we show that the first equation in part D holds. By the equations in part B, we have that

\[
a_1+\sqrt{d}=-1516739c_1+e_1+\sqrt{d}\in c_1\mathcal{O}_K+\left(e_1+\sqrt{d}\right)\mathcal{O}_K=P_1\textnormal{ and }
\]

\[
a_1+\sqrt{d}=-131890c_3-\left(e_3-\sqrt{d}\right)\in c_3\mathcal{O}_K+\left(e_3-\sqrt{d}\right)\mathcal{O}_K=\overline{P}_3.
\]

Furthermore, we have that $a_1+\sqrt{d}=-121339c_2^2-\left(2-\sqrt{d}\right)\in c_2^2\mathcal{O}_K+\left(2-\sqrt{d}\right)\mathcal{O}_K=\overline{P}_2^2$ by the equations in parts A and B. Along the same lines it can be proved that $a_1+\sqrt{d}\in P_5\cap\overline{P}_6\cap\overline{P}_7\cap P_{43}\cap\overline{P}_{53}$.

\smallskip
This implies that $a_1+\sqrt{d}\in P_1\cap\overline{P}_2^2\cap\overline{P}_3\cap P_5\cap\overline{P}_6\cap\overline{P}_7\cap P_{43}\cap\overline{P}_{53}$. Therefore, $\left(a_1+\sqrt{d}\right)\mathcal{O}_K\subseteq P_1\cap\overline{P}_2^2\cap\overline{P}_3\cap P_5\cap\overline{P}_6\cap\overline{P}_7\cap P_{43}\cap\overline{P}_{53}=P_1\overline{P}_2^2\overline{P}_3P_5\overline{P}_6\overline{P}_7P_{43}\overline{P}_{53}$ (since the powers of the maximal ideals in the aforementioned intersection are pairwise comaximal).

\medskip
We infer by the equations in part C that
\begin{align*}
{\rm N}\left(\left(a_1+\sqrt{d}\right)\mathcal{O}_K\right)=&|{\rm N}\left(a_1+\sqrt{d}\right)|=d-a_1^2=c_1c_2^2c_3c_5c_6c_7c_{43}c_{53}\\
=&{\rm N}\left(P_1\right){\rm N}\left(\overline{P}_2\right)^2{\rm N}\left(\overline{P}_3\right){\rm N}\left(P_5\right){\rm N}\left(\overline{P}_6\right){\rm N}\left(\overline{P}_7\right){\rm N}\left(P_{43}\right){\rm N}\left(\overline{P}_{53}\right)\\
=&{\rm N}\left(P_1\overline{P}_2^2\overline{P}_3P_5\overline{P}_6\overline{P}_7P_{43}\overline{P}_{53}\right).
\end{align*}

Therefore, $\left(a_1+\sqrt{d}\right)\mathcal{O}_K=P_1\overline{P}_2^2\overline{P}_3P_5\overline{P}_6\overline{P}_7P_{43}\overline{P}_{53}$.

\medskip
We want to emphasize that the equations in part D show that the maximal ideals $P_j$ and $\overline{P}_j$ for $j\in [1,s]$ are the only maximal ideals of $\mathcal{O}_K$ that can occur in the factorization of $\left(a_i+\sqrt{d}\right)\mathcal{O}_K$ into maximal ideals for each $i\in [1,r]$.

\medskip
Next, we determine the factorization of $z\mathcal{O}_K$ into maximal ideals of $\mathcal{O}_K$. Let $I$ be a nonzero ideal of $\mathcal{O}_K$ and let $P$ be maximal ideal of $\mathcal{O}_K$. In what follows, let ${\rm v}_P\left(I\right)$ be the unique nonnegative integer $k$ such that $I=P^kJ$ for some nonzero ideal $J$ of $\mathcal{O}_K$ with $J\nsubseteq P$. Note that the equations in part E below consist of two equalities. Each of the first equalities is an immediate consequence of the equations in part D, while each of the second equalities is (a consequence of) elementary arithmetic.

\medskip
\textbf{Equations, part E}
\medskip

{\tiny
${\rm v}_{P_1}\left(z\mathcal{O}_K\right)=b_1+b_7+b_8+b_9+b_{11}+b_{13}+b_{15}+b_{16}+b_{20}+b_{27}+b_{28}+b_{30}+b_{31}+b_{34}+b_{36}+b_{37}+b_{39}+b_{40}+b_{41}+b_{47}+b_{49}+$

\hspace*{15.2mm} $b_{51}+b_{53}+b_{54}+b_{57}=2d_1$,

${\rm v}_{P_2}\left(z\mathcal{O}_K\right)=b_4+3b_5+b_7+2b_9+b_{10}+2b_{14}+3b_{16}+2b_{18}+5b_{20}+2b_{21}+b_{23}+b_{26}+b_{27}+2b_{33}+b_{34}+b_{38}+b_{40}+b_{42}+b_{43}+$

\hspace*{15.2mm} $b_{45}+b_{52}=d_2$,

${\rm v}_{\overline{P}_2}\left(z\mathcal{O}_K\right)=2b_1+b_6+b_{12}+b_{13}+b_{15}+b_{17}+2b_{19}+b_{22}+b_{25}+b_{30}+b_{32}+b_{36}+2b_{41}+b_{44}+2b_{46}+b_{47}+b_{53}+b_{55}+2b_{57}=d_2$,

${\rm v}_{P_3}\left(z\mathcal{O}_K\right)=2b_2+b_3+b_5+b_6+2b_8+b_{12}+b_{15}+2b_{24}+b_{26}+b_{33}+b_{36}+b_{40}+b_{43}+b_{50}+b_{54}+b_{55}=d_3$,

${\rm v}_{\overline{P}_3}\left(z\mathcal{O}_K\right)=b_1+b_4+b_7+b_{11}+b_{13}+2b_{17}+b_{20}+2b_{27}+b_{32}+b_{42}+b_{47}+b_{49}+b_{51}+b_{52}+b_{53}+2b_{56}+b_{57}=d_3$,

${\rm v}_{P_4}\left(z\mathcal{O}_K\right)=b_8+b_{22}+b_{28}+b_{44}+2b_{46}+b_{48}+b_{53}=d_4$, ${\rm v}_{\overline{P}_4}\left(z\mathcal{O}_K\right)=b_{20}+b_{24}+b_{30}+2b_{37}+b_{45}+b_{54}+b_{56}=d_4$,

${\rm v}_{P_5}\left(z\mathcal{O}_K\right)=b_1+b_8+b_{12}+b_{16}+b_{23}+b_{25}+b_{40}+b_{48}=d_5$, ${\rm v}_{\overline{P}_5}\left(z\mathcal{O}_K\right)=b_7+b_{15}+b_{30}+b_{35}+b_{36}+b_{37}+b_{52}+b_{53}=d_5$,

${\rm v}_{P_6}\left(z\mathcal{O}_K\right)=b_9+b_{11}+b_{29}+b_{32}+b_{37}=d_6$, ${\rm v}_{\overline{P}_6}\left(z\mathcal{O}_K\right)=b_1+b_{31}+b_{33}+b_{38}+b_{41}+2b_{47}+b_{49}+b_{51}+b_{53}=d_6$,

${\rm v}_{P_7}\left(z\mathcal{O}_K\right)=b_4+b_{29}+b_{39}+2b_{50}+b_{56}=d_7$, ${\rm v}_{\overline{P}_7}\left(z\mathcal{O}_K\right)=b_1+b_2+b_{19}+b_{34}+b_{43}+b_{54}=d_7$, ${\rm v}_{P_8}\left(z\mathcal{O}_K\right)=b_{15}+b_{19}+b_{22}=d_8$,

${\rm v}_{\overline{P}_8}\left(z\mathcal{O}_K\right)=b_{24}+b_{33}+b_{50}=d_8$, ${\rm v}_{P_9}\left(z\mathcal{O}_K\right)=b_6+b_{14}+b_{25}+b_{29}+b_{50}=d_9$, ${\rm v}_{\overline{P}_9}\left(z\mathcal{O}_K\right)=b_2+b_{40}+b_{44}+b_{53}+b_{55}=d_9$,

${\rm v}_{P_{10}}\left(z\mathcal{O}_K\right)=b_{29}+b_{38}+b_{45}=d_{10}$, ${\rm v}_{\overline{P}_{10}}\left(z\mathcal{O}_K\right)=b_8+b_{13}+2b_{34}=d_{10}$, ${\rm v}_{P_{11}}\left(z\mathcal{O}_K\right)=b_{26}+b_{35}+b_{42}=d_{11}$, ${\rm v}_{\overline{P}_{11}}\left(z\mathcal{O}_K\right)=b_3+b_{54}=d_{11}$,

${\rm v}_{P_{12}}\left(z\mathcal{O}_K\right)=b_{35}+b_{37}+b_{54}=d_{12}$, ${\rm v}_{\overline{P}_{12}}\left(z\mathcal{O}_K\right)=b_4+b_5+b_7+b_{20}+b_{25}+b_{26}=d_{12}$, ${\rm v}_{P_{13}}\left(z\mathcal{O}_K\right)=b_{23}+b_{29}+b_{32}+b_{42}=d_{13}$,

${\rm v}_{\overline{P}_{13}}\left(z\mathcal{O}_K\right)=b_{19}+b_{53}=d_{13}$, ${\rm v}_{P_{14}}\left(z\mathcal{O}_K\right)=b_6+b_{10}+b_{39}=d_{14}$, ${\rm v}_{\overline{P}_{14}}\left(z\mathcal{O}_K\right)=b_{54}=d_{14}$, ${\rm v}_{P_{15}}\left(z\mathcal{O}_K\right)=b_{13}+b_{28}+b_{34}+b_{43}+b_{46}=d_{15}$,

${\rm v}_{\overline{P}_{15}}\left(z\mathcal{O}_K\right)=b_3+2b_{19}+2b_{38}=d_{15}$, ${\rm v}_{P_{16}}\left(z\mathcal{O}_K\right)=b_{29}=d_{16}$, ${\rm v}_{\overline{P}_{16}}\left(z\mathcal{O}_K\right)=b_2+b_{31}+b_{51}=d_{16}$, ${\rm v}_{P_{17}}\left(z\mathcal{O}_K\right)=b_{12}+b_{55}=d_{17}$,

${\rm v}_{\overline{P}_{17}}\left(z\mathcal{O}_K\right)=b_{30}+b_{31}+b_{56}=d_{17}$, ${\rm v}_{P_{18}}\left(z\mathcal{O}_K\right)=b_{16}+b_{25}+b_{27}=d_{18}$, ${\rm v}_{\overline{P}_{18}}\left(z\mathcal{O}_K\right)=b_{21}+b_{23}+b_{39}+b_{45}=d_{18}$,

${\rm v}_{P_{19}}\left(z\mathcal{O}_K\right)=b_{26}+b_{29}=d_{19}$, ${\rm v}_{\overline{P}_{19}}\left(z\mathcal{O}_K\right)=b_9+b_{21}+b_{33}=d_{19}$, ${\rm v}_{P_{20}}\left(z\mathcal{O}_K\right)=b_{39}=d_{20}$, ${\rm v}_{\overline{P}_{20}}\left(z\mathcal{O}_K\right)=2b_{22}=d_{20}$,

${\rm v}_{P_{21}}\left(z\mathcal{O}_K\right)=b_{31}+b_{44}+b_{46}=d_{21}$, ${\rm v}_{\overline{P}_{21}}\left(z\mathcal{O}_K\right)=b_{13}+2b_{18}+b_{41}=d_{21}$, ${\rm v}_{P_{22}}\left(z\mathcal{O}_K\right)=2b_{23}=d_{22}$, ${\rm v}_{\overline{P}_{22}}\left(z\mathcal{O}_K\right)=b_{28}+b_{52}+b_{57}=d_{22}$,

${\rm v}_{P_{23}}\left(z\mathcal{O}_K\right)=b_{14}+b_{16}+b_{37}+b_{44}=d_{23}$, ${\rm v}_{\overline{P}_{23}}\left(z\mathcal{O}_K\right)=b_6+b_{28}+b_{39}+b_{42}=d_{23}$, ${\rm v}_{P_{24}}\left(z\mathcal{O}_K\right)=b_{30}=d_{24}$, ${\rm v}_{\overline{P}_{24}}\left(z\mathcal{O}_K\right)=b_{50}=d_{24}$,

${\rm v}_{P_{25}}\left(z\mathcal{O}_K\right)=b_{13}+b_{28}+b_{40}=d_{25}$, ${\rm v}_{\overline{P}_{25}}\left(z\mathcal{O}_K\right)=b_{11}+b_{52}=d_{25}$, ${\rm v}_{P_{26}}\left(z\mathcal{O}_K\right)=b_{32}+b_{36}=d_{26}$, ${\rm v}_{\overline{P}_{26}}\left(z\mathcal{O}_K\right)=b_{20}+b_{24}=d_{26}$,

${\rm v}_{P_{27}}\left(z\mathcal{O}_K\right)=b_{46}=d_{27}$, ${\rm v}_{\overline{P}_{27}}\left(z\mathcal{O}_K\right)=b_2+b_{49}=d_{27}$, ${\rm v}_{P_{28}}\left(z\mathcal{O}_K\right)=b_{22}+b_{36}+b_{55}=d_{28}$, ${\rm v}_{\overline{P}_{28}}\left(z\mathcal{O}_K\right)=b_{12}+b_{45}+b_{48}=d_{28}$,

${\rm v}_{P_{29}}\left(z\mathcal{O}_K\right)=b_{21}+b_{27}=d_{29}$, ${\rm v}_{\overline{P}_{29}}\left(z\mathcal{O}_K\right)=b_{36}=d_{29}$, ${\rm v}_{P_{30}}\left(z\mathcal{O}_K\right)=b_{25}+b_{30}+b_{49}=d_{30}$, ${\rm v}_{\overline{P}_{30}}\left(z\mathcal{O}_K\right)=b_{14}+b_{17}+b_{34}+b_{41}+b_{47}=d_{30}$,

${\rm v}_{P_{31}}\left(z\mathcal{O}_K\right)=b_7+b_{17}=d_{31}$, ${\rm v}_{\overline{P}_{31}}\left(z\mathcal{O}_K\right)=b_{45}+b_{55}=d_{31}$, ${\rm v}_{P_{32}}\left(z\mathcal{O}_K\right)=b_{43}+b_{45}=d_{32}$, ${\rm v}_{\overline{P}_{32}}\left(z\mathcal{O}_K\right)=b_{36}=d_{32}$,

${\rm v}_{P_{33}}\left(z\mathcal{O}_K\right)=b_{18}+b_{27}+b_{44}=d_{33}$, ${\rm v}_{\overline{P}_{33}}\left(z\mathcal{O}_K\right)=b_{16}+b_{24}=d_{33}$, ${\rm v}_{P_{34}}\left(z\mathcal{O}_K\right)=b_{33}+b_{44}=d_{34}$, ${\rm v}_{\overline{P}_{34}}\left(z\mathcal{O}_K\right)=b_{15}+b_{17}=d_{34}$,

${\rm v}_{P_{35}}\left(z\mathcal{O}_K\right)=b_6=d_{35}$, ${\rm v}_{\overline{P}_{35}}\left(z\mathcal{O}_K\right)=b_{23}=d_{35}$, ${\rm v}_{P_{36}}\left(z\mathcal{O}_K\right)=b_{26}+2b_{31}=d_{36}$, ${\rm v}_{\overline{P}_{36}}\left(z\mathcal{O}_K\right)=b_{11}+b_{34}=d_{36}$,

${\rm v}_{P_{37}}\left(z\mathcal{O}_K\right)=b_4+b_{10}+b_{15}+b_{18}+b_{19}=d_{37}$, ${\rm v}_{\overline{P}_{37}}\left(z\mathcal{O}_K\right)=b_{11}=d_{37}$, ${\rm v}_{P_{38}}\left(z\mathcal{O}_K\right)=b_{43}=d_{38}$, ${\rm v}_{\overline{P}_{38}}\left(z\mathcal{O}_K\right)=b_7+b_{21}=d_{38}$,

${\rm v}_{P_{39}}\left(z\mathcal{O}_K\right)=b_5+b_9+b_{35}+b_{42}+b_{46}+b_{48}+b_{55}=d_{39}$, ${\rm v}_{\overline{P}_{39}}\left(z\mathcal{O}_K\right)=b_3+b_{38}+b_{53}=d_{39}$, ${\rm v}_{P_{40}}\left(z\mathcal{O}_K\right)=b_4+b_{47}+b_{57}=d_{40}$,

${\rm v}_{\overline{P}_{40}}\left(z\mathcal{O}_K\right)=b_{38}=d_{40}$, ${\rm v}_{P_{41}}\left(z\mathcal{O}_K\right)=b_{20}=d_{41}$, ${\rm v}_{\overline{P}_{41}}\left(z\mathcal{O}_K\right)=b_4=d_{41}$, ${\rm v}_{P_{42}}\left(z\mathcal{O}_K\right)=b_{52}=d_{42}$, ${\rm v}_{\overline{P}_{42}}\left(z\mathcal{O}_K\right)=b_{30}=d_{42}$,

${\rm v}_{P_{43}}\left(z\mathcal{O}_K\right)=b_1=d_{43}$, ${\rm v}_{\overline{P}_{43}}\left(z\mathcal{O}_K\right)=b_{16}=d_{43}$, ${\rm v}_{P_{44}}\left(z\mathcal{O}_K\right)=b_{13}+b_{51}=d_{44}$, ${\rm v}_{\overline{P}_{44}}\left(z\mathcal{O}_K\right)=b_{12}=d_{44}$, ${\rm v}_{P_{45}}\left(z\mathcal{O}_K\right)=b_{51}=d_{45}$,

${\rm v}_{\overline{P}_{45}}\left(z\mathcal{O}_K\right)=b_{18}+b_{40}=d_{45}$, ${\rm v}_{P_{46}}\left(z\mathcal{O}_K\right)=b_7+b_{10}+b_{32}+b_{33}+b_{35}+b_{49}+b_{52}=d_{46}$, ${\rm v}_{\overline{P}_{46}}\left(z\mathcal{O}_K\right)=b_2+b_9+b_{42}=d_{46}$,

${\rm v}_{P_{47}}\left(z\mathcal{O}_K\right)=b_6+b_{50}+b_{54}=d_{47}$, ${\rm v}_{\overline{P}_{47}}\left(z\mathcal{O}_K\right)=b_{43}+2b_{56}=d_{47}$, ${\rm v}_{P_{48}}\left(z\mathcal{O}_K\right)=b_9=d_{48}$, ${\rm v}_{\overline{P}_{48}}\left(z\mathcal{O}_K\right)=b_5+b_{12}=d_{48}$,

${\rm v}_{P_{49}}\left(z\mathcal{O}_K\right)=b_{27}+b_{41}=d_{49}$, ${\rm v}_{\overline{P}_{49}}\left(z\mathcal{O}_K\right)=b_{14}+b_{40}+b_{48}=d_{49}$, ${\rm v}_{P_{50}}\left(z\mathcal{O}_K\right)=b_{57}=d_{50}$, ${\rm v}_{\overline{P}_{50}}\left(z\mathcal{O}_K\right)=b_{24}+b_{49}=d_{50}$,

${\rm v}_{P_{51}}\left(z\mathcal{O}_K\right)=b_5+b_{22}=d_{51}$, ${\rm v}_{\overline{P}_{51}}\left(z\mathcal{O}_K\right)=b_{41}=d_{51}$, ${\rm v}_{P_{52}}\left(z\mathcal{O}_K\right)=b_3+b_{17}+b_{21}+b_{32}+b_{48}=d_{52}$, ${\rm v}_{\overline{P}_{52}}\left(z\mathcal{O}_K\right)=b_8+b_{10}=d_{52}$,

${\rm v}_{P_{53}}\left(z\mathcal{O}_K\right)=b_3+b_{11}=d_{53}$, ${\rm v}_{\overline{P}_{53}}\left(z\mathcal{O}_K\right)=b_1+b_{14}+b_{51}+b_{57}=d_{53}$, ${\rm v}_{P_{54}}\left(z\mathcal{O}_K\right)=b_8+b_{39}=d_{54}$, ${\rm v}_{\overline{P}_{54}}\left(z\mathcal{O}_K\right)=b_{25}=d_{54}$,

${\rm v}_{P_{55}}\left(z\mathcal{O}_K\right)=b_{15}+b_{28}=d_{55}$, ${\rm v}_{\overline{P}_{55}}\left(z\mathcal{O}_K\right)=b_{10}+b_{35}=d_{55}$, ${\rm v}_{P_{56}}\left(z\mathcal{O}_K\right)=b_{37}=d_{56}$, ${\rm v}_{\overline{P}_{56}}\left(z\mathcal{O}_K\right)=b_{26}+b_{47}=d_{56}$.

}

\medskip
The equations in part E show that ${\rm v}_{P_1}\left(z\mathcal{O}_K\right)=2d_1={\rm v}_{P_1}\left(n\mathcal{O}_K\right)$ and ${\rm v}_{P_j}\left(z\mathcal{O}_K\right)={\rm v}_{\overline{P}_j}\left(z\mathcal{O}_K\right)=d_j={\rm v}_{\overline{P}_j}\left(n\mathcal{O}_K\right)={\rm v}_{P_j}\left(n\mathcal{O}_K\right)$ for each $j\in [2,s]$, and thus $\eta\mathcal{O}_K=\mathcal{O}_K$. We infer that $\eta\in\mathcal{O}_K^{\times}$.\qed(Claim 2)

\medskip
\textsc{Claim} 3: $\eta\in\mathcal{O}_d^{\times}$.

\medskip
Observe that $d>c_j$ and $c_j\in\mathbb{P}$ for each $j\in [1,s]$. It follows that $d\nmid n$ (since $d\in\mathbb{P}$ by Claim 1). Since $z\in\mathcal{O}_K$, there are some $g,h\in\mathbb{Z}$ such that $z=g+h\sqrt{d}$. Moreover, since $\eta\in\mathcal{O}_K$ (by Claim 2) and $\{1,\sqrt{d}\}$ is a $\mathbb{Z}$-basis of $\mathcal{O}_K$, we obtain that $n\mid g$, $n\mid h$ and $\eta=\frac{g}{n}+\frac{h}{n}\sqrt{d}$. Clearly, $\mathcal{O}_K^{\times}\cap\mathcal{O}_d=\mathcal{O}_d^{\times}$. Since $d\nmid n$, $d\in\mathbb{P}$ and $\eta\in\mathcal{O}_K^{\times}$ (by Claim 2), we infer that $\eta\in\mathcal{O}_d^{\times}$ if and only if $\eta\in\mathcal{O}_d$ if and only if $d\mid\frac{h}{n}$ if and only if $d\mid h$ if and only if $z\in\mathcal{O}_d$. Consequently, it remains to show that $z\in\mathcal{O}_d$.

\smallskip
Set $z^{\prime}=\prod_{i=1}^r\left(a_i+b_i\sqrt{d}\right)$. Since $\left(u+\sqrt{d}\right)^v\equiv u^{v-1}\left(u+v\sqrt{d}\right)\mod d\mathcal{O}_K$ for all $u\in\mathbb{Z}$ and $v\in\mathbb{N}$, there is some $t\in\mathbb{Z}$ such that $z\equiv tz^{\prime}\mod d\mathcal{O}_K$. For this reason, it is sufficient to show that $z^{\prime}\in\mathcal{O}_d$. Next we determine $\prod_{i=1}^k\left(a_i+b_i\sqrt{d}\right)$ modulo $d\mathcal{O}_K$ step-by-step for each $k\in [2,r]$. Thereby, we use that $\left(v+w\sqrt{d}\right)\left(v^{\prime}+w^{\prime}\sqrt{d}\right)\equiv vv^{\prime}+\left(wv^{\prime}+vw^{\prime}\right)\sqrt{d}\mod d\mathcal{O}_K$ for each $v,v^{\prime},w,w^{\prime}\in\mathbb{Z}$. Note that the items below consist of two equations that give us the reductions of $u_k$ and $v_k$ modulo $d$ (for each $k\in [2,r]$), where $u_k,v_k\in\mathbb{Z}$ are such that $\prod_{i=1}^k\left(a_i+b_i\sqrt{d}\right)=u_k+v_k\sqrt{d}$.

\medskip

{\tiny
$\bullet$ $a_1a_2=6890530871592$,\hspace*{34.5mm} $b_1a_2+a_1b_2=-d+21554509784529$.

$\bullet$ $6890530871592a_3=-315587d+28244816980645$,\hspace*{4.9mm} $21554509784529a_3+6890530871592b_3=-972567d+5178868509861$.

$\bullet$ $28244816980645a_4=-1272650d+15513835265240$,\hspace*{2.3mm} $5178868509861a_4+28244816980645b_4=2697384d+4467486423321$.

$\bullet$ $15513835265240a_5=-675500d+35811027000980$,\hspace*{3.6mm} $4467486423321a_5+15513835265240b_5=652149d+30537129644761$.

$\bullet$ $35811027000980a_6=-1309781d+20714594938939$,\hspace*{2.3mm} $30537129644761a_6+35811027000980b_6=6236863d+16497568105021$.

$\bullet$ $20714594938939a_7=-441941d+21223637681572$,\hspace*{3.6mm} $16497568105021a_7+20714594938939b_7=28697d+18837954382492$.

$\bullet$ $21223637681572a_8=-407041d+13338881989779$,\hspace*{3.6mm} $18837954382492a_8+21223637681572b_8=4721606d+17458165870498$.

$\bullet$ $13338881989779a_9=-230639d+3661427934964$,\hspace*{4.9mm} $17458165870498a_9+13338881989779b_9=5070046d+17002573835617$.

$\bullet$ $3661427934964a_{10}=-52048d+15123266299360$,\hspace*{4.9mm} $17002573835617a_{10}+3661427934964b_{10}=130052d+34143299535672$.

$\bullet$ $15123266299360a_{11}=-202084d+22867483171996$,\hspace*{2.3mm} $34143299535672a_{11}+15123266299360b_{11}=6208881d+10414810162433$.

$\bullet$ $22867483171996a_{12}=-282499d+8083910734589$,\hspace*{3.6mm} $10414810162433a_{12}+22867483171996b_{12}=7832662d+17492146626900$.

$\bullet$ $8083910734589a_{13}=-92186d+37797607362521$,\hspace*{4.9mm} $17492146626900a_{13}+8083910734589b_{13}=1575020d+23141664104777$.

$\bullet$ $37797607362521a_{14}=-411842d+28944089196110$,\hspace*{2.3mm} $23141664104777a_{14}+37797607362521b_{14}=1939038d+6609389335501$.

$\bullet$ $28944089196110a_{15}=-239446d+16720583049264$,\hspace*{2.3mm} $6609389335501a_{15}+28944089196110b_{15}=2568152d+34718516820815$.

$\bullet$ $16720583049264a_{16}=-120376d+6912281518232$,\hspace*{3.6mm} $34718516820815a_{16}+16720583049264b_{16}=1211229d+10227378309170$.

$\bullet$ $6912281518232a_{17}=-44790d+5195414457466$,\hspace*{6.2mm} $10227378309170a_{17}+6912281518232b_{17}=741048d+289652693936$.

$\bullet$ $5195414457466a_{18}=-30152d+22465854532740$,\hspace*{4.9mm} $289652693936a_{18}+5195414457466b_{18}=265509d+17460040771835$.

$\bullet$ $22465854532740a_{19}=-123878d+18660863010082$,\hspace*{2.3mm} $17460040771835a_{19}+22465854532740b_{19}=2001631d+19162392372661$.

$\bullet$ $18660863010082a_{20}=-99441d+31621822715453$,\hspace*{3.6mm} $19162392372661a_{20}+18660863010082b_{20}=1834171d+26774899887696$.

$\bullet$ $31621822715453a_{21}=-167069d+18945559551357$,\hspace*{2.3mm} $26774899887696a_{21}+31621822715453b_{21}=4282685d+14809237415832$.

$\bullet$ $18945559551357a_{22}=-93807d+25490759313159$,\hspace*{3.6mm} $14809237415832a_{22}+18945559551357b_{22}=1035530d+16642252647821$.

$\bullet$ $25490759313159a_{23}=-82380d+31410566401668$,\hspace*{3.6mm} $16642252647821a_{23}+25490759313159b_{23}=5180714d+7665156542464$.

$\bullet$ $31410566401668a_{24}=-97899d+24750475711101$,\hspace*{3.6mm} $7665156542464a_{24}+31410566401668b_{24}=4386896d+12345706811136$.

$\bullet$ $24750475711101a_{25}=-34513d+4341133124105$,\hspace*{4.9mm} $12345706811136a_{25}+24750475711101b_{25}=8807546d+33705248273090$.

$\bullet$ $4341133124105a_{26}=-5385d+20420906899575$,\hspace*{6.2mm} $33705248273090a_{26}+4341133124105b_{26}=888730d+19138455539665$.

$\bullet$ $20420906899575a_{27}=-21329d+11628413963866$,\hspace*{3.6mm} $19138455539665a_{27}+20420906899575b_{27}=432256d+14371072675706$.

$\bullet$ $11628413963866a_{28}=-10322d+27536499676132$,\hspace*{3.6mm} $14371072675706a_{28}+11628413963866b_{28}=2230673d+20483449818325$.

$\bullet$ $27536499676132a_{29}=-16214d+11871313871146$,\hspace*{3.6mm} $20483449818325a_{29}+27536499676132b_{29}=10491877d+23841147446157$.

$\bullet$ $11871313871146a_{30}=1168d+36708968638606$,\hspace*{6.9mm} $23841147446157a_{30}+11871313871146b_{30}=583202d+29824187122383$.

$\bullet$ $36708968638606a_{31}=27670d+35292993605984$,\hspace*{5.6mm} $29824187122383a_{31}+36708968638606b_{31}=6287182d+30828612734725$.

$\bullet$ $35292993605984a_{32}=33791d+32099367907223$,\hspace*{5.6mm} $30828612734725a_{32}+35292993605984b_{32}=2928174d+15667131309718$.

$\bullet$ $32099367907223a_{33}=31502d+8686497691888$,\hspace*{6.9mm} $15667131309718a_{33}+32099367907223b_{33}=1722251d+13069916858845$.

$\bullet$ $8686497691888a_{34}=17842d+32178145784978$,\hspace*{6.9mm} $13069916858845a_{34}+8686497691888b_{34}=1025368d+38021671372563$.

$\bullet$ $32178145784978a_{35}=68542d+38082283375434$,\hspace*{5.6mm} $38021671372563a_{35}+32178145784978b_{35}=5937837d+11732682959923$.

$\bullet$ $38082283375434a_{36}=100145d+35965182833267$,\hspace*{4.3mm} $11732682959923a_{36}+38082283375434b_{36}=6202247d+37211608563518$.

$\bullet$ $35965182833267a_{37}=106375d+33169240097520$,\hspace*{4.3mm} $37211608563518a_{37}+35965182833267b_{37}=8936948d+26020836879171$.

$\bullet$ $33169240097520a_{38}=99259d+17708055977779$,\hspace*{5.6mm} $26020836879171a_{38}+33169240097520b_{38}=10150666d+32544641385778$.

$\bullet$ $17708055977779a_{39}=70781d+30770548150680$,\hspace*{5.6mm} $32544641385778a_{39}+17708055977779b_{39}=2202939d+36915064492287$.

$\bullet$ $30770548150680a_{40}=155955d+11904176271315$,\hspace*{4.3mm} $36915064492287a_{40}+30770548150680b_{40}=2562988d+165258901717$.

$\bullet$ $11904176271315a_{41}=64624d+25508831080099$,\hspace*{5.6mm} $165258901717a_{41}+11904176271315b_{41}=1347305d+5907415462461$.

$\bullet$ $25508831080099a_{42}=143454d+872140076252$,\hspace*{6.9mm} $5907415462461a_{42}+25508831080099b_{42}=790654d+34225185690977$.

$\bullet$ $872140076252a_{43}=5137d+10987048845161$,\hspace*{9.5mm} $34225185690977a_{43}+872140076252b_{43}=339647d+26138731918415$.

$\bullet$ $10987048845161a_{44}=74345d+8152289751913$,\hspace*{6.9mm} $26138731918415a_{44}+10987048845161b_{44}=1871482d+38377823723539$.

$\bullet$ $8152289751913a_{45}=56296d+17399000858672$,\hspace*{6.9mm} $38377823723539a_{45}+8152289751913b_{45}=295750d+32817075202061$.

$\bullet$ $17399000858672a_{46}=140874d+13424492118010$,\hspace*{4.3mm} $32817075202061a_{46}+17399000858672b_{46}=2189576d+33163475056782$.

$\bullet$ $13424492118010a_{47}=116689d+22068166729399$,\hspace*{4.3mm} $33163475056782a_{47}+13424492118010b_{47}=705436d+24070914532262$.

$\bullet$ $22068166729399a_{48}=266604d+20874064459588$,\hspace*{4.3mm} $24070914532262a_{48}+22068166729399b_{48}=292240d+15318085049644$.

$\bullet$ $20874064459588a_{49}=267012d+25167892634080$,\hspace*{4.3mm} $15318085049644a_{49}+20874064459588b_{49}=789145d+30753673516637$.

$\bullet$ $25167892634080a_{50}=384789d+20538358542749$,\hspace*{4.3mm} $30753673516637a_{50}+25167892634080b_{50}=1701637d+23065592907429$.

$\bullet$ $20538358542749a_{51}=591526d+21236595367747$,\hspace*{4.3mm} $23065592907429a_{51}+20538358542749b_{51}=3306393d+31259701520158$.

$\bullet$ $21236595367747a_{52}=628169d+663951808753$,\hspace*{6.9mm} $31259701520158a_{52}+21236595367747b_{52}=1963739d+37377022803680$.

$\bullet$ $663951808753a_{53}=20667d+11352564579816$,\hspace*{8.2mm} $37377022803680a_{53}+663951808753b_{53}=1565314d+35652369557844$.

$\bullet$ $11352564579816a_{54}=408337d+12677902910401$,\hspace*{4.3mm} $35652369557844a_{54}+11352564579816b_{54}=6095135d+24252482581547$.

$\bullet$ $12677902910401a_{55}=632652d+37493335194470$,\hspace*{4.3mm} $24252482581547a_{55}+12677902910401b_{55}=2876153d+28715685383254$.

$\bullet$ $37493335194470a_{56}=1912828d+27624385992608$,\hspace*{3mm} $28715685383254a_{56}+37493335194470b_{56}=11211960d+33470792628624$.

$\bullet$ $27624385992608a_{57}=1505206d+34038147456710$,\hspace*{3mm} $33470792628624a_{57}+27624385992608b_{57}=6487920d$.

}

\medskip
Consequently, $z^{\prime}\equiv 34038147456710\mod d\mathcal{O}_K$, and thus $z^{\prime}\in\mathcal{O}_d$.\qed(Claim 3)

\medskip
\textsc{Claim} 4: $\eta\not=1$.

\medskip
First, we show that $z\equiv\sqrt{d}\mod 3\mathcal{O}_K$. We have that $d\equiv 2\mod 3\mathcal{O}_K$ and $\left(0+\sqrt{d}\right)^2\equiv\left(1+\sqrt{d}\right)^4\equiv\left(2+\sqrt{d}\right)^4\equiv 2\mod 3\mathcal{O}_K$. Moreover, $\left(0+\sqrt{d}\right)^4\equiv\left(1+\sqrt{d}\right)^8\equiv\left(2+\sqrt{d}\right)^8\equiv 1\mod 3\mathcal{O}_K$. Using this, we infer that

{\tiny
\begin{align*}
z\equiv &\left(0+\sqrt{d}\right)^3\left(2+\sqrt{d}\right)^6\left(0+\sqrt{d}\right)^2\left(1+\sqrt{d}\right)^3\left(2+\sqrt{d}\right)^6\left(0+\sqrt{d}\right)^2\left(0+\sqrt{d}\right)^0\left(1+\sqrt{d}\right)^2\left(0+\sqrt{d}\right)^3\left(2+\sqrt{d}\right)^6\left(2+\sqrt{d}\right)^2\left(0+\sqrt{d}\right)^1\\
&\left(2+\sqrt{d}\right)^5\left(2+\sqrt{d}\right)^7\left(2+\sqrt{d}\right)^1\left(1+\sqrt{d}\right)^3\left(2+\sqrt{d}\right)^2\left(2+\sqrt{d}\right)^3\left(0+\sqrt{d}\right)^1\left(2+\sqrt{d}\right)^3\left(1+\sqrt{d}\right)^7\left(0+\sqrt{d}\right)^3\left(1+\sqrt{d}\right)^2\left(1+\sqrt{d}\right)^4\\
&\left(1+\sqrt{d}\right)^0\left(0+\sqrt{d}\right)^3\left(1+\sqrt{d}\right)^3\left(0+\sqrt{d}\right)^3\left(0+\sqrt{d}\right)^1\left(0+\sqrt{d}\right)^3\left(1+\sqrt{d}\right)^7\left(0+\sqrt{d}\right)^1\left(1+\sqrt{d}\right)^2\left(1+\sqrt{d}\right)^3\left(1+\sqrt{d}\right)^2\left(0+\sqrt{d}\right)^2\\
&\left(1+\sqrt{d}\right)^7\left(2+\sqrt{d}\right)^1\left(1+\sqrt{d}\right)^6\left(2+\sqrt{d}\right)^5\left(1+\sqrt{d}\right)^5\left(2+\sqrt{d}\right)^3\left(2+\sqrt{d}\right)^3\left(1+\sqrt{d}\right)^1\left(1+\sqrt{d}\right)^7\left(2+\sqrt{d}\right)^0\left(0+\sqrt{d}\right)^0\left(1+\sqrt{d}\right)^4\\
&\left(1+\sqrt{d}\right)^2\left(2+\sqrt{d}\right)^7\left(0+\sqrt{d}\right)^0\left(2+\sqrt{d}\right)^7\left(0+\sqrt{d}\right)^2\left(2+\sqrt{d}\right)^6\left(2+\sqrt{d}\right)^7\left(2+\sqrt{d}\right)^6\left(2+\sqrt{d}\right)^6\\
\equiv &\left(0+\sqrt{d}\right)^2\left(1+\sqrt{d}\right)^6\left(2+\sqrt{d}\right)^4\equiv 2^3\left(1+\sqrt{d}\right)^2\equiv\sqrt{d}\mod 3\mathcal{O}_K.
\end{align*}
}

Since $z\equiv\sqrt{d}\mod 3\mathcal{O}_K$, it follows that $z\not\in\mathbb{Z}$, and hence $\eta\not=1$.\qed(Claim 4)

\medskip
It remains to show that $d\mid y$ (where $y$ was defined at the beginning of Section~\ref{3}). Observe that for each $j\in [1,r]$, $1\leq -3033477+\sqrt{d}\leq a_j+\sqrt{d}\leq 2126573+\sqrt{d}\leq 2^{100}$. Moreover, $\max\{b_i\mid i\in [1,r]\}=23621570\leq 10^8$ and $r\leq 100$, and thus

\[
\eta\leq\prod_{i=1}^r\left(a_i+\sqrt{d}\right)^{b_i}\leq\prod_{i=1}^r\left(2^{100}\right)^{10^8}=2^{10^{10}r}\leq 2^{10^{12}}<2^d\leq\left(1+\sqrt{d}\right)^d\leq\varepsilon^d.
\]

Also note that for each $j\in [1,s]$, $c_j\leq 2^{10}$. Furthermore, $\max\{d_i\mid i\in [1,s]\}=146634276\leq 10^9$ and $s\leq 100$, and hence

\[
\eta^{-1}\leq\prod_{j=1}^s c_j^{d_j}\leq\prod_{j=1}^s\left(2^{10}\right)^{10^9}=2^{10^{10}s}\leq 2^{10^{12}}<2^d\leq\left(1+\sqrt{d}\right)^d\leq\varepsilon^d.
\]

It follows from Claims 2 and 4 that $\eta\in\mathcal{O}_K^{\times}\setminus\{1\}$. Since $\varepsilon^{-d}<\eta<\varepsilon^d$, there is some $k\in\mathbb{Z}\setminus\{0\}$ such that $-d<k<d$ and $\eta=\varepsilon^k$. Set ${\rm ord}\left(\varepsilon\mathcal{O}_d^{\times}\right)=\min\{m\in\mathbb{N}\mid\left(\varepsilon\mathcal{O}_d^{\times}\right)^m=\mathcal{O}_d^{\times}\}$. Observe that ${\rm ord}\left(\varepsilon\mathcal{O}_d^{\times}\right)=\left(\mathcal{O}_K^{\times}:\mathcal{O}_d^{\times}\right)\mid d$ (because $\mathcal{O}_K^{\times}/\mathcal{O}_d^{\times}$ is generated by $\varepsilon\mathcal{O}_d^{\times}$ and $d\in\mathbb{P}$ is ramified in $\mathcal{O}_K$). Since $\left(\varepsilon\mathcal{O}_d^{\times}\right)^k=\eta\mathcal{O}_d^{\times}=\mathcal{O}_d^{\times}$ by Claim 3, we have that ${\rm ord}\left(\varepsilon\mathcal{O}_d^{\times}\right)\mid k$, and thus $\left(\mathcal{O}_K^{\times}:\mathcal{O}_d^{\times}\right)=1$ (since $d\in\mathbb{P}$). We infer that $\varepsilon\in\mathcal{O}_d$, and hence $d\mid y$.\qed(Theorem~\ref{Theorem 1.3})

\bigskip
\noindent {\bf ACKNOWLEDGEMENTS.}
We want to thank A. Geroldinger for helpful remarks and comments that improved this note.

\end{document}